\documentclass[12pt]{amsart}

\usepackage[utf8]{inputenc}
\usepackage{amssymb}
\usepackage{mathrsfs}
\usepackage{xspace}
\usepackage{enumerate}
\usepackage{cancel}
\usepackage{ifthen}
\usepackage[plain]{fancyref}
\usepackage{hyperref}
\usepackage[normalem]{ulem}
\usepackage{mathrsfs}
\usepackage{comment}
\usepackage[ firstinits=true, maxbibnames=99, natbib=true, hyperref=true, ]{biblatex}
\renewbibmacro{in:}{}
\newtheorem{thm}{Theorem}[section]
\newtheorem{lem}[thm]{Lemma}
\newtheorem{cor}[thm]{Corollary}
\newtheorem{prop}[thm]{Proposition}
\newtheorem{fact}[thm]{Fact}
\newtheorem{claim}[thm]{Claim}

\newtheorem*{claim*}{Claim}

\theoremstyle{definition}
\newtheorem{defn}[thm]{Definition}

\newtheorem{question}[thm]{Question}
\newtheorem{conj}[thm]{Conjecture}

\theoremstyle{remark}
\newtheorem{rem}[thm]{Remark}

\newcommand*\fancyrefthmlabelprefix{thm}\frefformat{plain}{\fancyrefthmlabelprefix}{Theorem~#1}
\newcommand*\fancyreflemlabelprefix{lem}\frefformat{plain}{\fancyreflemlabelprefix}{Lemma~#1}
\newcommand*\fancyrefproplabelprefix{prop}\frefformat{plain}{\fancyrefproplabelprefix}{Proposition~#1}
\newcommand*\fancyrefcorlabelprefix{cor}\frefformat{plain}{\fancyrefcorlabelprefix}{Corollary~#1}
\newcommand*\fancyrefclaimlabelprefix{claim}\frefformat{plain}{\fancyrefclaimlabelprefix}{Claim~#1}
\newcommand*\fancyreffactlabelprefix{fact}\frefformat{plain}{\fancyreffactlabelprefix}{Fact~#1}
\newcommand*\fancyrefquestionlabelprefix{question}\frefformat{plain}{\fancyrefquestionlabelprefix}{Question~#1}
\newcommand*\fancyrefconjlabelprefix{conj}\frefformat{plain}{\fancyrefconjlabelprefix}{Conjecture~#1}
\newcommand*\fancyrefdefnlabelprefix{defn}\frefformat{plain}{\fancyrefdefnlabelprefix}{Definition~#1}
\newcommand*\fancyrefconstlabelprefix{const}\frefformat{plain}{\fancyrefconstlabelprefix}{Construction~#1}
\newcommand*\fancyrefsetuplabelprefix{setup}\frefformat{plain}{\fancyrefsetuplabelprefix}{Setup~#1}
\newcommand*\fancyrefexlabelprefix{ex}\frefformat{plain}{\fancyrefexlabelprefix}{Example~#1}
\newcommand*\fancyrefremlabelprefix{rem}\frefformat{plain}{\fancyrefremlabelprefix}{Remark~#1}
\newcommand*\fancyrefsubseclabelprefix{subsec}\frefformat{plain}{\fancyrefsubseclabelprefix}{subsection~#1}
\frefformat{plain}{\fancyrefeqlabelprefix}{(#1)}
\newcommand*\fancyrefitemlabelprefix{item}\frefformat{plain}{\fancyrefitemlabelprefix}{(#1)}

\def\repeat#1#2 {\expandafter\gdef\csname B#1\endcsname {\mathbb{#1}}
  \ifthenelse{\equal{#2}{*}}{}{\repeat #2 }}
\repeat ABCDEFGHIJKLMNOPQRSTUVWXYZ*
\def\repeat#1#2 {\expandafter\gdef\csname C#1\endcsname {\mathcal{#1}}
  \ifthenelse{\equal{#2}{*}}{}{\repeat #2 }}
\repeat ABCDEFGHIJKLMNOPQRSTUVWXYZ*
\def\repeat#1#2 {\expandafter\gdef\csname bf#1\endcsname {\boldsymbol{#1}}
  \ifthenelse{\equal{#2}{*}}{}{\repeat #2 }}
\repeat abcdefghijklmnopqrstuvwxyzABCDEFGHIJKLMNOPQRSTUVWXYZ*
\def\repeat#1#2 {\expandafter\gdef\csname scr#1\endcsname {\mathscr{#1}}
  \ifthenelse{\equal{#2}{*}}{}{\repeat #2 }}
\repeat abcdABCDEFGHIJKLMNOPQRSTUVWXYZ*

\allowdisplaybreaks

\newcommand{\xxx}[2][] {%
  \ifthenelse{\equal{#1}{}}
  {\underline{$\bullet$}}
  {\uline{#1}}\marginpar{\tiny #2}\xspace}

\let\epsilon\varepsilon

\DeclareMathOperator{\im}{im}
\DeclareMathOperator{\rank}{rank}
\DeclareMathOperator{\Hom}{Hom}
\DeclareMathOperator{\Aut}{Aut}

\DeclareMathOperator{\GL}{GL}

\let\Ordo\CO
\newcommand\Mat[1][n]{{\rm Mat}^{#1\times #1}_q}
\newcommand\NNN{\scrN^{n\times n}_q}
\newcommand\NNNs{\tilde\scrN^{n\times n}_q}
\newcommand\RRR{\scrR^{n\times n}_q}

\begin{document}

\title[{A proof of the Liebeck--Nikolov--Shalev conjecture}]{Initiating the proof of the Liebeck--Nikolov--Shalev conjecture}

\author[Gill]{Nick Gill}
\address{Nick Gill,
  School of Mathematics and Statistics, The Open University,
  Walton Hall, Milton Keynes, MK7 6AA, United Kingdom}
\email{nick.gill@open.ac.uk}

\author[Lifshitz]{Noam Lifshitz}
\address{Noam Lifshitz,
  Einstein Institute of Mathematics,
  The Hebrew University of Jerusalem
  Givat Ram. Jerusalem, 9190401, Israel} 
\email{noamlifshitz@gmail.com}

\def\hungaryaddress{{A.~R\'enyi Institute of Mathematics,
    HUN-REN,
    P.O.~Box~127, H-1364 Budapest, Hungary}}
\author[Pyber]{L\'aszl\'o Pyber}
\address{L\'aszl\'o Pyber, \hungaryaddress}
\email{pyber@renyi.hu}

\author[Szab\'o]{Endre Szab\'o}
\address{Endre Szab\'o, \hungaryaddress}
\email{endre@renyi.hu}

\thanks{
  LP, and ESz were supported by the National Research,
  Development and Innovation Office (NKFIH) Grant K138596,
  The project leading to this application has received funding from
  the European Research Council (ERC) under the European Union's
  Horizon 2020 research and innovation programme (grant agreement No
  741420)
  NL is supported by ISF grant  1980/22 and by a European Research Council (ERC) 2024 starting grant (grant agreement No
  101163794).
}

\date{\today}

\begin{abstract}
   Liebeck, Nikolov, and Shalev conjectured that for every subset of a finite simple group $A\subseteq S$ with  $|A|\ge 2$, there exist $O\left(\frac{\log |S|}{\log |A|}\right)$ conjugates of $A$ whose product is $S$. 
   This paper is a companion to \cite{Lifshitz-AgB-L2-Flattening} and together they prove the conjecture. In this paper we prove the conjecture in the regime where $|A|>|S|^{c}$ for an absolute constant $c>0$.

  We also prove that the following Skew Product Theorem holds for all finite simple groups.
  Namely we show that either the product of two conjugates of $A$ has size
  at least $|A|^{1.49}$,
  or $S$ is the product of boundedly many conjugates of $A$.
\end{abstract}

\maketitle

\section{Introduction}
\label{sec:introduction}
The following is a deep result of Liebeck and Shalev
\cite{Liebeck.Shalev.2001}.

\begin{thm}
  \label{thm:Liebeck-Shalev}
  There exists a constant $C>0$ such that if $S$ is a non-abelian
  finite simple group and $N\subseteq S$ is a nontrivial normal
  subset, then
  \[
    N^m = S
    \quad\text{for any \ }m \ge C\log|S|/\log|N|.
  \]
\end{thm}

Liebeck, Nikolov and Shalev \cite{lns2},
motivated by \fref{thm:Liebeck-Shalev},
and several other results concerning product decompositions of finite
simple groups,
proposed the following.

\begin{conj}
  \label{conj:Liebeck-Nikolov-Shlev}
  There exists an absolute constant $C>0$ such that if $S$ is a ﬁnite simple group
  and $A$ is any subset of $S$ of size at least 2,
  then $S$ is a product of $N$ conjugates of $A$ for some
  $N \le C \log |S|/ \log |A|$.
\end{conj}

This paper is a part of a pair that together complete the proof of the Liebeck--Nikolov--Shalev conjecture. Here we prove the conjecture for sets of size at least $|S|^{c}$ for an absolute constant $c>0$:
\begin{thm}
  \label{thm:Filling-G-ad-delta-sized-sets}
  For every $\delta>0$ there is a $M=M(\delta)>0$ with the following property.
  Let $S$ be a non-abelian finite simple group,
  and $A_1,\dots,A_M$ be subsets of size at least $|S|^\delta$.
  Then there is a product decomposition
  \[
    S = A_1^{\sigma_1}\cdots A_M^{\sigma_M}
    \quad
    \text{for some }
    \sigma_1,\ldots,\sigma_M \in S.
  \]
\end{thm}

While preparing this manuscript, Theorem \ref{thm:Filling-G-ad-delta-sized-sets} was mentioned to Dona~\cite{dona2024writing} who managed to combine it with a novel combinatorial argument to make progress towards the Liebeck--Nikolov--Shalev conjecture. Namely, he showed that for every $\epsilon >0$ there exists $N_{\epsilon}>0$, such that if $A\subseteq S$ is a subset of a finite simple group with $|A|\ge 2$, then there exist $N_{\epsilon}\left(\frac{\log|S|}{\log |A|}\right)^{1+\epsilon}$ conjugates of either $A$ or $A^{-1}$ whose product is $S$. 

In a subsequent work, Lifshitz~\cite{Lifshitz-AgB-L2-Flattening} complements our result to complete the proof of the Nikolov--Liebeck--Shalev conjecture by showing that there exist absolute constants $c,C>0$, such that for every subet $A$ of a finite simple group $S$ and every integer $m\ge \frac{C\log |S|}{\log |A|}$, there exist $m$ conjugates of $A$ whose product has size at least $|S|^{c}$. 

It is worth noting that the proof of Dona \cite{dona2024writing} is purely combinatorial while the later proof of Lifshitz~\cite{Lifshitz-AgB-L2-Flattening} is character theoretic in nature.

\subsection{Product theorems}
The Product Theorem for finite simple groups of Lie type of bounded
rank proved independently in \cite{bgt2} and \cite{pyber2016growth}
states the following.

\begin{thm}
  \label{thm:Product-Theorem}
  Let $L$ be a ﬁnite simple group of Lie type of rank $r$ and $A$ a generating
  set of $L$. Then either
  $|A^3 | > |A|^{1+\epsilon}$,
  or  $A^3 = L$,
  where $\epsilon$ depends only on $r$.
\end{thm}

By a lemma of Petridis \cite{Petridis2012Pluenecke-inequalities}
and Tao \cite{tao2008product-set-estimates-noncommutative}
this is equivalent to the following

\begin{thm}
  \label{thm:Product-Theorem-reformulated}
  Let $L$ be a ﬁnite simple group of Lie type of rank $r$ and $A$ a generating
  set of $L$. Then either
  $|AA^a | > |A|^{1+\epsilon}$ for some $a\in A$,
  or $A^3 = L$,
  where $\epsilon$ depends only on $r$.
\end{thm}

There are several families of examples
\cite{ps2, Gill.Pyber.Short.Szabo.2011}
that show that $\epsilon$ in these theorems must depend on the rank of
$L$.
In particular, in \cite{ps2}
generating sets $A$ of $SL(n, 3)$ of size $2n-1 + 4$ are constructed,
with $|A^3 | < 100|A|$, i.e., generating sets of
constant growth.

In spite of these examples, we have been able to find a meaningful analogue of the
Product Theorem, or more precisely,
of \fref{thm:Product-Theorem-reformulated},
that is valid for simple groups of unbounded rank. 

\begin{thm}[Skew Product Theorem]
  \label{thm:Skew-Product-Theorem} For every $\delta>0$ there exists $M\in \mathbb{N}$, such that the following holds.
  Let $S$ be a non-abelian ﬁnite simple group
  and $A$ a subset of size at least two.
  Then either
  $|AA^\sigma| \ge |A|^{1.5-\delta}$ for some $\sigma\in S$,
  or $A^{\sigma_1}\cdots A^{\sigma_M} = S$
  for some $\sigma_1,\dots,\sigma_M\in S$.
\end{thm}

In the bounded rank regime, this theorem already follows from known results, see Section \ref{sec:bounded-rank}. Here we prove this result for groups of unbounded rank, where we make use of the following notion of rank. 

\begin{defn}
  Throughout this paper,
  finite simple groups are assumed to be nonabelian.
  We define the rank of a finite simple group to be its untwisted Lie rank if it
  is a group of Lie type, and to be its degree if it is an alternating group.
\end{defn}

Our proof of the Skew Product Theorem and Theorem \ref{thm:Filling-G-ad-delta-sized-sets} for high rank finite simple groups comes in two parts.
In the first part we prove a generalization of the statement
for sets which are ``not too large'' or very small. For the former case the statement is as follows. 

\begin{thm}
  \label{thm:Skew-Product-Theorem-two-sets}
  For every $\epsilon>0$ there is a $\delta>0,r\in\mathbb{N}$
  with the following property.
  Let $S$ be a non-abelian finite simple group of rank at least $r$,
  and $B$, $A$ be subsets such that $|A|,|B|\le|S|^{1-\epsilon}$ and also
  \(|A|^{1+\delta}\ge|B|\ge|A|^{1-\delta}.\)
  Then
  \[
    |BA^{\sigma}| \ge |B|\cdot|A|^{\delta}
    \quad
    \text{for some }\sigma\in S.
  \]
\end{thm}

For smaller sets we obtain the following improved bound.

\begin{thm}
  \label{thm:Skew-Product-Theorem-effective}
  For every $\epsilon>0$ there exists $\delta>0$, such that if $S$ is a non-abelian ﬁnite simple group
  and $A,B\subseteq S$ are with $|A|,|B|\le |S|^{\delta}$. Then 
  $|A^{\sigma}B| \ge |A|^{1.5 -\epsilon}|B|^{1-\epsilon}$ for some $\sigma\in S$.
\end{thm}

For the second part, we also prove a generalization of the statement, this time for sets which are ``large enough''.

\begin{thm}
  \label{thm:Filling-with-very-large-sets}
  There exist absolute constants $\delta>0$ and $M\in \mathbb{N}$
  with the following property.
  Let $S$ be a non-abelian finite simple group, and
  $A_1,\dots,A_M$ be subsets of $S$
  of size at least $|S|^{1-\delta}$.
  Then there is a product decomposition
  $$
  S = A_1^{\sigma_1}\cdots A_M^{\sigma_M}
  \quad\quad
  \text{for some }\sigma_1,\dots,\sigma_M\in S.
  $$
\end{thm}

 When $G$ is a finite simple group  $M$ can in fact be taken to be 6 (See Theorem \ref{thm: filling large lie types} below for the high rank regime and Nikolov--Pyber~\cite{nikolov2011product} for the bounded rank regime.)

\subsection{Method}
The results of this paper follow as a corollary of Thorems \ref{thm:Skew-Product-Theorem-two-sets}-\ref{thm:Filling-with-very-large-sets}
when combined with previous works. The tools used in the proofs of Theorems 
\ref{thm:Skew-Product-Theorem-two-sets}
and \ref{thm:Skew-Product-Theorem-effective}  make use of the probabilistic method in the form of a second moment argument, while the proof of \fref{thm:Filling-with-very-large-sets}
requires a completely different character theoretic approach
related to \emph{hypercontractivity for global functions}. We also give a second proof of \fref{thm:Filling-with-very-large-sets} for groups of Lie type bases on Fourier analysis in abelian groups.  
\subsection*{Our combinatorial argument for Theorems \ref{thm:Skew-Product-Theorem-two-sets} and \ref{thm:Skew-Product-Theorem-effective}}
Both of these theorems were proven in the special case where $A$ is a conjugacy class (see Propositions \ref{prop:product.with.normal.set.grows} and \ref{prop:product.with.normal.set.grows.fast} due to Gill, Pyber, Short, and Szab\'{o}~\cite{Gill.Pyber.Short.Szabo.2011} and Dona, Mar\'{o}ti, and Pyber\cite{Product-of-subsets-in-classical-groups}). It is not too difficult to extend this results to the case where $A$ is a large subset of a conjugacy class or more generally a translate of such set. We then complete the proof by making use of a second moment argument to show that either $|A^{\sigma}B|$ grows for some $\sigma\in S$ or $A$ contains a translate of large subset of a conjugacy class. 

\subsection*{A character theoretic criterion for product decompositions}

A subset $A\subseteq G$ is said to be \emph{normal} if it is a union of conjugacy classes. Character bounds have been used extensively to study the growth of normal sets within finite simple groups (see e.g. \cite{guralnick2020character,larsen2008characters,Liebeck.Shalev.2001,shalev}.) One of the tools in this paper is a character theoretic criterion for product decompositions, which holds even for non-normal sets $A$. 

\begin{thm}\label{thm:character theoretic criterion for product decompositions in finite simple groups}
    There exists an absolute constant $C>0$, such that the following holds. Let $\epsilon>0,$ and $M>\frac{C}{\epsilon}$ be an even integer. Suppose that $S$ is a finite simple group and $A_1,\ldots, A_M\subseteq S$ are subsets, such that for all $i$ and all irreducible characters $\chi$ we have
    \[
    \left|\frac{1}{|A_i|^2}\sum_{a\in A_i}\sum_{b\in A_i}\chi(a^{-1}b)\right|<2 \chi(1)^{1-\epsilon}.
    \]
    Then there exists $\sigma_1,\ldots, \sigma_M$, such that $A_1^{\sigma_1}\cdots A_M^{\sigma_M} = S.$
\end{thm}

\subsection*{The nonabelian Fourier decomposition of a function}
We shall often identify a representation $(V,\rho)$ of a finite group $G$ with the corresponding irreducible character $\chi_{\rho} = \mathrm{tr}\circ \rho.$ Given $v\in V$ and a functional $\varphi\in V^{*},$ the function on $G$ given by $\sigma \mapsto \varphi(\sigma v)$ is called a \emph{matrix coefficient} for the representation $\rho$ and we denote the space spanned by the matrix coefficients of $\rho$ by $W_{\rho}$ or $W_{\chi}$ if $\chi = \mathrm{tr}\circ \rho$.

The group algebra $\mathbb{C}[G]$ of complex valued functions on $G$ is equipped with the inner product $\langle f, g\rangle = \frac{1}{|G|}\sum_{\sigma \in G}f(\sigma)\overline{g(\sigma)},$
and we write $\|f\|_2^2 = \langle f,f\rangle.$ The group algebra can be orthogonally decomposed as the direct sum of the spaces $W_{\chi}$, and we write $f^{=\chi}$ is the projection of $f$ onto the space $W_{\chi}$.

\subsection*{Fourier anti-concentration inequalities}
Our Fourier theoretic approach is to analyse a set $A$ via bounds on $\|f^{=\chi}\|_2$ for the normalized indicator function of $A$ given by $f = \frac{|G|}{|A|}1_A.$ Such bounds are known as Fourier anti-concentration inequalities.  
As shown in \eqref{eq:convolve} below we have $\| f^{=\chi} \|_2^2 = \frac{\chi(1)}{|A|^2}\sum_{a,b\in A}[\chi(ab^{-1})],$ and therefore the hypothesis in Theorem \ref{thm:character theoretic criterion for product decompositions in finite simple groups} can be equivalently expressed as an upper bound on $\|f^{=\chi}\|_2.$ In light of this our goal becomes to find criterions on a set $A$ that imply that $\|f^{=\chi}\|_2^2\le 2\chi(1)^{2-c}$ for a constant $c>0$. In the case of finite simple groups of Lie type we establish this for all sets $A\subseteq S$  with $|A|>|S|^{1-\delta}$ for a sufficiently small absolute constant $\delta>0$ by appealing to the deep character bounds of Guralnick, Larsen, and Tiep~\cite{guralnick2024character}, which they obtained via Deligne--Lustig theory. 

\subsection*{Anti-concentration inequalities for global functions}
For alternating groups our approach is related to hypercontractivity for global functions.
Given a subset $I\subseteq [n]$ we denote the set of permutations in $A_n$ that fix every point in $I$ by $U_I$. A \emph{$t$-umvirate} is a coset $\sigma U_I,$ for a set $I$ of size $t$. A subset $A\subseteq A_n$ is said to be $r$-global if $\frac{|A\cap U|}{|U|}\le r^{t}\frac{|A|}{|A_n|}$ for every $t$-umvirate $U$. (The analogue notion is defined for $S_n$ by Keevash and Lifshitz~\cite{keevash2023sharp}, but the difference between $A_n$ and $S_n$ is not too important.) 

As shown in \cite[Proposition 1.9]{keller2023improved},  the Fourier anti-concentration inequalities of Keevash and Lifshitz~\cite{keevash2023sharp} imply that for every $\epsilon>0$ there exists $\delta>0$, such that if $\alpha>0$, $A\subseteq S_n$ has size $\ge e^{-n^{\alpha}}n!$ and $A$ is $n^{\delta}$-global, then the function $f= \frac{|G|}{|A|}1_A$ satisfies  
\[\|f^{=\chi}\|_2 \le \chi(1)^{\alpha+\epsilon},\]  for every irreducible character $\chi$. We establish the following result in the spirit of \cite{keevash2023sharp}.  

\begin{thm}\label{thm:anti-concentration inequalities for global functions}
    For every $\epsilon>0$ there exists $\delta>0$, such that if a nonempty set $A\subseteq A_n$ is $n^{1-\epsilon}$-global and $f= \frac{|A_n|}{|A|}1_A.$ Then 
    \[\|f^{=\chi}\|_2^2\le 2 \chi(1)^{2-\delta}\] for every irreducible character $\chi$ of $A_n$.
\end{thm}

Note that Theorem \ref{thm:anti-concentration inequalities for global functions} has the significantly weaker hypothesis that $A$ is $n^{1-\epsilon}$-global, rather than $n^{\delta}$-global like in \cite[Proposition 1.3]{keller2023improved}. The conclusion of Theorem \ref{thm:anti-concentration inequalities for global functions} is weaker provided that $|A| \ge e^{-n^{0.9}}n!$, say, and stronger when $|A|\le e^{-n}n!$.

\subsection*{Our alternative proof of Theorem \ref{thm:Filling-with-very-large-sets} for groups of Lie type}
We also present an alternative proof for Theorem \ref{thm:Filling-with-very-large-sets} in the case of finite simple groups of Lie type. 

Our idea is to use the fact that all such groups $S$  contain a `large' abelian group $(N, +)$ with a large normalizer $L$. Given a subset $A\subseteq S$ a simple averaging argument shows that one can then find $\sigma\in G$ with $\frac{|\sigma^{-1}A \cap N|}{|N|} \ge \frac{|A|}{|G|}$. We then use Fourier analysis of the abelian group $N$ to show that for every large subset $B\subseteq N$ there exist few elements $\sigma_1,\ldots, \sigma_M$ of $L$ with $B^{\sigma_1} + \cdots + B^{\sigma_M} = N$. 

This allows us to find few conjugates of the set $\sigma^{-1}A$ that contain the large abelian group $N$. We then show that there are few conjugates of $N$ whose product is $G$ to deduce that there are few conjugates of $A$ whose product is $G$. 

\subsection{Open problems}

\fref{thm:Skew-Product-Theorem-two-sets} motivates the following.

\begin{conj}[Asymmetric growth conjecture]
  \label{conj:product-grows-or-very-large}
  For every $\delta>0$ there exists $\epsilon>0$
  such that
  for any non-abelian finite simple group $S$
  and subsets $B$ and $A$ of $S$ with
  $|B|\le|S|^{1-\delta}$ 
  we have
  $$
  |BA^g|\ge|B|\cdot|A|^\epsilon
  \quad
  \text{for some }g\in S.
  $$
\end{conj}

In \cite[Proposition~5.2]{Gill.Pyber.Short.Szabo.2011}
this conjecture is shown to be true, if $A$ is a normal subset.

In a subsequent work \cite[Theorem~1.4]{Lifshitz-AgB-L2-Flattening}
Lifshitz
proved this conjecture for finite simple groups of Lie type, while making use of our generalized Frobenius formula. The conjecture remains open for alternating groups.

One could hope for an even stronger result for somewhat smaller subsets.

\begin{conj}[Fast growth]
  \label{conj:product-grows-fast-or-large}
  For every $\epsilon>0$ there exists $\delta>0$
  such that
  for any non-abelian finite simple group $S$
  and subsets $B$ and $A$ of $S$ with
  $|B|,|A|\le|S|^\delta$ 
  we have
  $$
  |BA^g|\ge|B|\cdot|A|^{1-\epsilon}
  \quad
  \text{for some }g\in S.
  $$  
\end{conj}
By \fref{cor:orthogonal-conjugate-bounded-rank}
the conjecture holds for finite simple group of Lie type of bounded rank.
For classical groups of rank $r$ over the field of $q$ elements,
\fref{conj:product-grows-fast-or-large}
follows with $\epsilon=0$
if $B$ and $A$ have size at most $q^{r/4}$,
see \fref{prop:small-sets-grow-fast}.
Surprisingly, this conjecture is also known to hold when $A$ is a normal
subset, see \cite{Product-of-subsets-in-classical-groups}.

\fref{thm:Filling-G-ad-delta-sized-sets} implies this conjecture for
very large subsets $A$.
For an account of other partial results see \cite{lns2}.

A strengthening of \fref{conj:Liebeck-Nikolov-Shlev},
namely, \cite[Conjecture 1]{Gill.Pyber.Short.Szabo.2011},
is confirmed for large subsets
by \fref{thm:Filling-G-ad-delta-sized-sets}.
 
In a subsequent work Lifshitz~\cite{Lifshitz-AgB-L2-Flattening} proves the Liebeck--Nikolov--Shalev conjecture in full, while making use of both Theorem \ref{thm:Filling-G-ad-delta-sized-sets} and our generalized Frobenius formula. 
 
In  \cite{Gill.Pyber.Short.Szabo.2011} it was also shown that
\fref{conj:Liebeck-Nikolov-Shlev} implies a weaker form of the
Skew Product Theorem.
On the other hand, note that
\fref{conj:product-grows-or-very-large}
and the Skew Product Theorem together
imply \fref{conj:Liebeck-Nikolov-Shlev}.

\subsection{Structure of the paper}
\label{sec:structure-paper}

In our Preliminaries Section (Section \ref{sec:bounded-rank}) we show how Theorems \ref{thm:Filling-G-ad-delta-sized-sets}, \ref{thm:Skew-Product-Theorem}, \ref{thm:Skew-Product-Theorem-two-sets}, \ref{thm:Skew-Product-Theorem-effective} and \ref{thm:Filling-with-very-large-sets}, can be reduced to proving only Theorems \ref{thm:Skew-Product-Theorem-two-sets}-\ref{thm:Filling-with-very-large-sets} in the high rank regime.

In Section \ref{sec:proof-Skew-Product-Theorem-variants} we prove the high rank case of Theorems \ref{thm:Skew-Product-Theorem-two-sets} and \ref{thm:Skew-Product-Theorem-effective}. In
Section \ref{sec:Filling-very-large-Alternating} we prove Theorems \ref{thm:Filling-with-very-large-sets}-\ref{thm:anti-concentration inequalities for global functions}.

Finally, in Section \ref{sec:Filling-very-large-Lie-type} we give our alternative proof of \fref{thm:Filling-with-very-large-sets} for groups of Lie type.

\section{Preliminaries}
\subsection{Simple groups of bounded order}\label{sec:bounded-rank}

\begin{defn}
  Fix arbitrary constants $r>0$, $M>0$,
  and let $S$ be a nonabelian finite simple group.
  We distinguish three cases:
  \begin{enumerate}[\indent(a)]
  \item \label{item:bounded-size}
    $|S|\le M$,
  \item \label{item:Lie-bounded-rank}
    $S$ is a simple group of Lie type of rank at most $r$,
  \item \label{item:unbounded-rank}
    $S$ is either an alternating group or a simple group of Lie type,
    and the rank of $S$ is larger than $r$.
  \end{enumerate}
\end{defn}

The classification of finite simple groups imply the following.
\begin{fact}
  If $M$ is the maximum of the order of $A_{r-1}$ and the order of the
  Monster simple group, then the above three cases cover all finite
  simple groups.
\end{fact}

In this section we will prove
most of our results
for the class of groups satisfying
\fref{item:bounded-size} or \fref{item:Lie-bounded-rank}.
The only exception is \fref{thm:Skew-Product-Theorem-two-sets},
which is proved here only for the class of groups satisfying
\fref{item:bounded-size}.

\begin{lem}
  \label{lem:there-is-some-growth}
  Let $S$ be a finite simple group,
  $B$ a proper subset,
  and $A$ a subset of size at least $2$.
  Then
  \[
    \big|BA^{\sigma}\big| > |B|
    \quad\quad
    \text{for some }\sigma \in S.
  \]
\end{lem}
\begin{proof}
  Suppose, that the statement is false for certain subsets $B$ and $A$.
  Fix an element $b\in A$, then $Bb=BA$, hence $B=B(Aa^{-1})$.
  Similarly,
  \[
    B = B (Aa^{-1})^{\sigma}
    \quad\quad\text{for all }\sigma \in G.
  \]
  This implies that
  \[
    B = B H
  \]
  where $H$ is the subgroup generated by the conjugates of $Aa^{-1}$.
  Since $|Aa^{-1}|=|A|\ge2$, $H$ must contain a nontrivial conjugacy
  class. Therefore $H=G$, and so $B=BG=G$, a contradiction.
\end{proof}

\begin{cor}
  \label{cor:there-is-some-growth}
  Theorems~\ref{thm:Filling-G-ad-delta-sized-sets}, \ref{thm:Skew-Product-Theorem}, and 
  \ref{thm:Filling-with-very-large-sets} all
  hold under the additional hypothesis that $|S|\le M.$
\end{cor}
\begin{proof}
  Let $S$ be a simple group of size at most $N$,
  and $A_1,\dots,A_M$ be nonempty subsets of size $\ge 2$.
  Iterated use of \fref{lem:there-is-some-growth}
  implies that there are $\sigma_2,\dots,\sigma_M \in S$
  such that either
  \[
    \text{either}\quad
    A_1^{\sigma_1}\cdots A_M^{\sigma_M} = S
    \quad\text{or}\quad
    \big|A_1^{\sigma_1}\cdots A_M^{\sigma_M}\big| \ge M+1.
  \]
  So the Theorems \ref{thm:Skew-Product-Theorem}
  and \ref{thm:Skew-Product-Theorem-effective}
  hold with $M=N$,
  \fref{thm:Filling-with-very-large-sets}
  and \fref{thm:Filling-G-ad-delta-sized-sets}
  also hold with $=M$.
\end{proof}

\begin{lem}
  \label{lem:there-is-some-exponential-growth-in-bounded-groups}
  Let $M>0$ and $\epsilon = \frac{\log M - \log(M-1)}{\log M}$. Let $S$ be a simple group of size at most $M$,
  and $B,A$ be subsets with $B\ne S$. Then
  \[
    |BA^{\sigma}| \ge
    |B|\cdot|A|^\epsilon
    \quad
    \text{for some }\sigma \in S.    
  \]
\end{lem}
\begin{proof}
  If $B=\emptyset$ or $|A|\le1$ then the statement holds with
  $\sigma=1$.
  Otherwise  \fref{lem:there-is-some-growth}
  gives us an element $\sigma\in S$ such that
  \[
    |BA^{\sigma}| \ge
    |B|+1 \ge
    |B|\cdot\frac{M}{M-1} =
    |B|\cdot M^\epsilon \ge
    |B|\cdot|A|^\epsilon.
  \]
\end{proof}

\begin{cor}
  \label{cor:there-is-some-exponential-growth-in-bounded-groups}
  The variants of Theorems~\ref{thm:Skew-Product-Theorem-two-sets} and \ref{thm:Skew-Product-Theorem-effective} where $\delta$ is allowed to be sufficiently small with respect to $|S|$ hold.
\end{cor} 
\begin{proof}
The variant of Theorem \ref{thm:Skew-Product-Theorem-effective} follows immediately from Lemma \ref{lem:there-is-some-exponential-growth-in-bounded-groups}.
While the variant of Theorem \ref{thm:Skew-Product-Theorem-effective} holds since $A,B$ must have size $1$, provided that $\delta$ is sufficiently small and then the statement holds trivially.
\end{proof}

\subsection{Simple groups of bounded rank}
In \cite{Gill.Pyber.Short.Szabo.2011},
in the course of proving Lemma~2.1,
the following statement is proved.
\begin{prop}
  \label{prop:small-sets-grow-fast}
  Let $S$ be a non-abelian finite simple group and let $\kappa$ denote the size of the
  smallest nontrivial conjugacy class.
  If $B$ and $A$ are subsets of $S$ of size at most $\sqrt[4]{\kappa}$
  then
  \[
    |BA^{\sigma}|=|B|\cdot|A|
    \quad\quad
    \text{for some }\sigma \in S.
  \]
\end{prop}

Using the fact that the size of any nonidentity conjugacy class is polynomial in the size of a bounded rank finite simple group we may deduce the following.

\begin{cor}
  \label{cor:orthogonal-conjugate-bounded-rank}
  Let $\delta=\frac{1}{32r}$.
  If $S$ is a finite simple group of Lie type of rank at most $r$,
  and $B,A$ are subsets of size at most $|S|^\delta$
  then
  \[
    |BA^{\sigma}| = |B|\cdot|A|
    \quad\quad
    \text{for some }g\in S.
  \]
  In particular, Theorem \ref{thm:Skew-Product-Theorem-effective} holds when the rank is bounded and $\delta$ is sufficiently small with respect to the rank. 
\end{cor}
\begin{proof}
  Every nontrivial conjugacy class in $S$ has at least $|S|^{4\delta}$
  elements, provided that $\delta$ is sufficiently small (see
  e.g. \cite[Proposition~2.3]{Gill.Pyber.Short.Szabo.2011}).
  Then \fref{prop:small-sets-grow-fast}
  implies the statement.
\end{proof}

The main result of
\cite{Gill-Pyber-Szabo-2019-Roger-Saxl-bounded-rank}
is the following.
\begin{prop}
  \label{prop:4-authors-filling-in-bounded-rank}
  Let $S$ be a ﬁnite simple group of Lie type of rank at most $r$.
  There exists $C = f (r)>0$ such that if $A_1,\dots,A_M$ are subsets
  of $S$ satisfying
  $\prod_{i=1}^M |A_i|\ge|G|^C$, then there exist
  elements $\sigma_1,\dots,\sigma_M$ such that $G = A_1^{\sigma_1}\cdots A_M^{\sigma_M}$.
\end{prop}

We may now deduce our main results for finite simple groups of bounded rank.
  
\begin{cor}\label{cor:filling bounded rank}
  For finite simple group of Lie type of rank at most $r$,
  Theorems~\ref{thm:Filling-G-ad-delta-sized-sets} and \ref{thm:Filling-with-very-large-sets}
  hold with the constants $M$ depending on both $\delta$ and $r$.
\end{cor}
\begin{proof}
  Follows immediately from \fref{prop:4-authors-filling-in-bounded-rank}.
\end{proof}

\begin{cor}\label{cor:skew-product low rank}
  For finite simple group of Lie type of rank at most $r$,
  Theorem~\ref{thm:Skew-Product-Theorem}
  holds with $M$ depending on both $\delta,r$.
\end{cor}
\begin{proof}
  \fref{cor:orthogonal-conjugate-bounded-rank} gives us a constant
  $\delta$.
  For subsets of size $|A|\le|S|^\delta$
  then \fref{cor:orthogonal-conjugate-bounded-rank}
  implies the statements,
  and for subsets of size $|A|>|S|^\delta$
  \fref{prop:4-authors-filling-in-bounded-rank}
  implies the statement.
\end{proof}

\subsection{Reducing to Theorems \ref{thm:Skew-Product-Theorem-two-sets}-\ref{thm:Filling-with-very-large-sets}}

We now prove that Theorems \ref{thm:Filling-G-ad-delta-sized-sets} and \ref{thm:Skew-Product-Theorem} follow from Theorems \ref{thm:Skew-Product-Theorem-two-sets}-\ref{thm:Filling-with-very-large-sets}. 

\begin{proof}[Proof of Theorem \ref{thm:Filling-G-ad-delta-sized-sets} given that Theorems \ref{thm:Skew-Product-Theorem-two-sets} and \ref{thm:Filling-with-very-large-sets} hold]
 Let $I$ be the set of $t\in(0,1)$ for which there exists an integer $M = M(t)>0$, such that for all finite simple groups $S$ and every subsets $A_1,\ldots, A_M\subseteq S$  with $|A_i|>|S|^{t}$, we have $A_1^{\sigma_1}\cdots A_M^{\sigma_M} = S$ for some $\sigma_i\in S.$ By Theorem \ref{thm:Filling-with-very-large-sets} $I$ contains some number $<1.$ Let $t_0$ be the infimum of the set $I.$ 
 Our proof will be complete once we show that $t_0 = 0$. Suppose that $t>0$, and let $\delta' = \delta'(t)$ be sufficiently small. We now reach a contradiction by showing the existence of an integer $M$, such that if $S$ is a finite simple group and $A_1,\ldots, A_M\subseteq S$ have size $\ge |S|^{t_0-\delta'}$, then there exist $\sigma_i\in S$ with $A_1^{\sigma_1}\cdots A_M^{\sigma_M} = S$. 
 
 Provided that $\delta'\le t_0/2$, Theorem \ref{thm:Skew-Product-Theorem-two-sets} with $\epsilon = \frac{1+t_0}{2}$ now gives us the existence of $\delta,r$ depending only on $t_0$, such that if $S$ is of rank $\ge r$ and $A,B\subseteq S$ have size $\ge |S|^{t_0 - \delta'}$, then there exists $\sigma\in S$ with $|A^{\sigma}B| \ge |S|^{t+\delta_0}$. By Corollaries \ref{cor:filling bounded rank} and \ref{cor:there-is-some-growth} there exists $M' = M'(t/2),$ such that if $S$ either has rank $\le r$ or is sporadic,  and $A_1,\ldots ,A_{M'}\subseteq S$ have size $\ge |S|^{t_0-\delta'},$ then there exist $\sigma_1,\ldots, \sigma_{M'}\in S$ with 
 $A_1^{\sigma_1}\cdots A_{M'}^{\sigma_{M'}} = S$
 
 We set $M = M(t_0- \delta') = \max(2M(t_0+\delta'), M').$ To reach a contradiction we may assume that the rank of $S$ is $>r$ and therefore $M' = 2(M(t_0+\delta'))$. Let $A_1,\ldots, A_{2M(t_0+\delta')}\subseteq S$. By Theorem \ref{thm:Skew-Product-Theorem-two-sets} for all odd $i$ there exist $\sigma_i$ with $|A_{i}^{\sigma_i}A_{i+1}|\ge |S|^{t_0+\delta'}.$ By the definition of $M(t_0+\delta'),$ there exist conjugates of the sets $A_{i}^{\sigma_i}A_{i+1}$ whose product is $S$. This completes the proof of the assertion that Theorem \ref{thm:Filling-G-ad-delta-sized-sets} follows from Theorems \ref{thm:Skew-Product-Theorem-two-sets} and \ref{thm:Filling-with-very-large-sets}.
\end{proof}

We now prove that Theorem \ref{thm:Skew-Product-Theorem} follows from Theorems \ref{thm:Filling-G-ad-delta-sized-sets} and \ref{thm:Skew-Product-Theorem-effective}. 

\begin{proof}[Proof of Theorem \ref{thm:Skew-Product-Theorem} given that Theorems \ref{thm:Filling-G-ad-delta-sized-sets} and \ref{thm:Skew-Product-Theorem-effective} hold]
    By Theorem \ref{thm:Skew-Product-Theorem-effective} there exists $\delta = \delta(\epsilon/2)$ for which every sets $A,B\subseteq S$ of size $\le |S|^{\delta}$ satisfy $|A^{\sigma}B|\ge |A|^{1.5-\epsilon/2}|B|^{-\epsilon/2}.$ This implies the statement when $|A|<|S|^{\delta}$. By Theorem \ref{thm:Filling-G-ad-delta-sized-sets} the statement holds when $|A|>|S|^{\delta}.$
\end{proof}

\section{Growth for medium sized sets}\label{sec:proof-Skew-Product-Theorem-variants}

In this section we prove Theorems \ref{thm:Skew-Product-Theorem-two-sets} and \ref{thm:Skew-Product-Theorem-effective}. 
Both of these results concernlower bounds for $|A^{\sigma}B|$ and are known to hold when $A$ is a normal sets. These results can be easily generalized to the case where $A$ contains a translate of a `large' subset of a conjugacy class. 
Our proof proceeds by showing that a lower bound for $|A^{\sigma}B|$ holds unless $A$ contains such a set. 

\subsection{The special case where $A$ is normal}
We start by recalling the results for the case where $A$ is normal. 
\begin{prop}[{\cite[Proposition~5.2]{Gill.Pyber.Short.Szabo.2011}}]
  \label{prop:product.with.normal.set.grows}
  For every $\epsilon>0$ there exists
  $\delta\in(0,1]$
  such that for any
  non-abelian finite simple group $S$ and subsets $N$ and $B$ of $S$ with $N$
  normal in $S$ and $|B|\leq |S|^{1-\epsilon}$ we have
  $$
  |NB|\geq |N|^{\delta}|B|.
  $$
\end{prop}

We have the following stronger result when $N,B$ are assumed to be small.

\begin{prop}[{\cite[Theorem~1.1]{Product-of-subsets-in-classical-groups}}]
  \label{prop:product.with.normal.set.grows.fast}
  For any $\epsilon > 0$ there exists $\delta > 0$ such that
  if $S$ is a finite simple non-abelian group,
  $N\subseteq S$ is a normal subset and $B\subseteq S$ is a subset with $|N|, |B| \le |S|^\delta$,
  then $|NB| \ge |N|^{1-\epsilon}|B|$.
\end{prop}

\subsection{A second moment argument showing that $A$ is dense inside a translated conjugacy class}
The following lemma morally shows that either $|A^{\sigma}B|$ is large, or there exist translates of $A$ and $B^{-1}$ that are both correlated with the same conjugacy class $\alpha.$
\begin{lem}\label{lem: growth or concentration in a conjugacy class}
  Let $G$ be a group, let $A,B\subseteq G$, 
  and let $\mathcal{C}$ be the set of conjugacy classes of $G$.  Write  \[\Gamma = \sum_{\alpha\in \mathcal{C}, a\in A, b\in B} \frac{|\alpha \cap a^{-1}A||\alpha^{-1} \cap b^{-1} B|}{|\alpha||A||B|}\]
  Then there exists $\sigma \in G$ with 
$|A^{\sigma} B| \ge \frac{|A||B|}{\Gamma}.$
\end{lem}
This lemma was essentially proved in \cite{Lifshitz-AgB-L2-Flattening}, and we give the proof here for completeness. 
\begin{proof}
 Choose $\sigma,\tau$ randomly and independently from $G$. For each $a\in A,b\in B$ let $X_{a,b}$ be the indicator of the event $a\sigma b = \tau.$ Let $Y$ be the maximum of the $X_{a,b}$ over all $a\in A,b\in B$. Then $Y$ is the indicator of the event $\tau \in A\sigma B.$ 
 
 Taking expectation first over $\sigma$ and then over $\tau$ we obtain $\mathbb{E}[Y] = \mathbb{E}_{\sigma}\left[\frac{|A\sigma B|}{|G|}\right] = \mathbb{E}_{\sigma}\left[\frac{|A^{\sigma} B|}{|G|}\right].$ By the first moment method, in order to prove the lemma, it suffices to lower bound $\mathbb{E}[Y] \ge \frac{|A||B|}{\Gamma|G|}.$
 Let $X = \sum X_{a,b}.$ Then we may use the Payley--Zigmund inequality to lower bound $\mathbb{E}[Y] = \Pr[X>0] \ge \frac{\mathbb{E}^2[X]}{\mathbb{E}[X^2]}.$ 
 Let us start by computing first moments. Here we have $\mathbb{E}[X] = \frac{|A||B|}{|G|}$ as for each coice of $a,\sigma ,b$ there is only a single choice of $\tau$ with $\tau =a\sigma b.$  

 The second moment computation relies on the computation of expectations of the form $\mathbb{E}[X_{a,b}X_{a'b'}]$. 
 Now the conditional expectation $\mathbb{E}[X_{a,b}X_{a'b'} | \sigma]$ is 0 whenever $a\sigma b\ne a'\sigma b'$ and $\frac{1}{|G|}$ otherwise. Moreover, $a\sigma b =a'\sigma b'$ if and only if $\sigma^{-1} a^{-1} a' \sigma  = b b'^{-1}.$ This can happen only if $a^{-1}a'$ and $bb'^{-1}$ are in the same conjugacy class $\alpha$, and the equality then happens with probability $\frac{1}{|\alpha|}.$ 
 Fixing $a\in A,b\in B$,
 this shows that 
 \[
 \sum_{a',b'} \mathbb{E}[X_{a,b}X_{a',b'}] =  \sum_{\alpha \in \mathcal{C}, a' \in  a \alpha \cap A, b' \in b\alpha^{-1}\cap B} \frac{1}{|G||\alpha|} = \sum_{\alpha \in \mathcal{C}} \frac{|a\alpha \cap A| |b\alpha^{-1} \cap B|}{|\alpha||G|}.\]
Hence, 
\[
\mathbb{E}[X^2] = \sum_{a,b} \sum_{a',b'} \mathbb{E}[X_{a,b}X_{a',b'}] \le \sum_{a\in A,b\in B, \alpha\in \mathcal{C}} \frac{|a\alpha \cap A| |b\alpha^{-1} \cap B|}{|\alpha||G|} = \Gamma \frac{|A||B|}{|G|}.
\]
Hence, 
  The Payley--Zigmund inequality now implies that \[\mathbb{E}[Y] = \Pr[X>0] \ge \frac{\mathbb{E}^2[X]}{\mathbb{E}[X^2]} \ge \frac{|A||B|}{\Gamma|G|},\]
  which completes the proof of the lemma. 
\end{proof}

Lemma \ref{lem: growth or concentration in a conjugacy class} can be simplified for finite simple groups of high rank by making use of the following fact.

\begin{prop}[{\cite[Proposition~2.1]{Liebeck.Schul.Shalev.2017.rapid.growth}}]
  \label{prop:conjugacy-classes-grow-rapidly}
  For any $\epsilon > 0$ there exists $r=r(\epsilon)$ such that
  if $S$ is a finite simple group of rank at least $r$,
  then for all $m\in\BN$, $S$ has at most $m^\epsilon$
  conjugacy classes of size at most $m$.
\end{prop} 

We are now ready to prove our simplified version of Lemma \ref{lem: growth or concentration in a conjugacy class}.  
\begin{lem}\label{lem:second moment simplified}
    For every $\epsilon>0$ there exists $r>0$, such that the following holds. Let $G$ be a finite simple group of rank $\ge r$, let $A,B\subseteq G$, and $\Gamma > 1$. Suppose that $|A^{\sigma}B| \le \frac{|A||B|}{\Gamma}$ for all $\sigma\in G.$ Then there exists a conjugacy class $\alpha$ and $a\in A,b\in B$ with 
    \[
    \frac{|\alpha \cap a^{-1}A||\alpha^{-1} \cap b^{-1} B|}{|\alpha|} \ge \frac{\Gamma - 3|A|^{\epsilon}|B|^{\epsilon}}{|A|^{\epsilon}|B|^{\epsilon}}.
    \]
\end{lem}
\begin{proof}
    Suppose that $|A^{\sigma}B|\le \frac{|A||B|}{\Gamma}$ for all $\sigma\in G.$ Then by Lemma \ref{lem: growth or concentration in a conjugacy class} we have $\sum_{\alpha\in \mathcal{C}, a\in A, b\in B} \frac{|\alpha \cap a^{-1}A||\alpha^{-1} \cap b^{-1} B|}{|\alpha||A||B|} \ge \Gamma$ and therefore there exist $a\in A,b\in B$ with 
    \[
    \sum_{\alpha\in \mathcal{C}} \frac{|\alpha \cap a^{-1}A||\alpha^{-1} \cap b^{-1} B|}{|\alpha|}\ge \Gamma.
    \]

    We also have 
 \[
    \sum_{\alpha\in \mathcal{C},|\alpha|\ge |A||B|} \frac{|\alpha \cap a^{-1}A||\alpha^{-1} \cap b^{-1} B|}{|\alpha|}\le \sum_{\alpha\in \mathcal{C},|\alpha|\ge |A||B|} \frac{|A||B|}{|\alpha|}.
    \]
    By Proposition \ref{prop:conjugacy-classes-grow-rapidly} for every fixed $\delta>0$ there are at most $m^{\delta}$ conjugacy classes of size $\le m$, provided that the rank $r$ is sufficiently large. Counting the contribution of $\alpha$ with $|\alpha|$ in the interval $(|A||B|2^{i},|A||B|2^{i+1})$, it follows that
    \[
    \sum_{\alpha\in \mathcal{C},|\alpha|\ge |A||B|} \frac{|A||B|}{|\alpha|} \le \sum_{i=0}^{\infty} |A|^{\delta}|B|^{\delta} 2^{(i+1)\delta - i}\le 3|A|^{\delta}|B^{\delta}|\le 3|A|^{\epsilon}|B|^{\epsilon},
    \]
    provided that $\delta$ is sufficiently small.

    Thus,
    \[
    \sum_{\alpha\in \mathcal{C},|\alpha|\le |A||B|}  \frac{|A||B|}{|\alpha|} \ge \Gamma - 3|A|^{\epsilon}|B^{\epsilon}|
    \]
    By Proposition \ref{prop:conjugacy-classes-grow-rapidly} there are at most $|A|^{\epsilon}|B|^{\epsilon}$ conjugacy classes of size $\le |A||B|,$ provided that $r$ is sufficiently large. The lemma follows.
\end{proof}

We are now ready to further simplify and prove that either $|A^{\sigma}B|$ grows or $A$ contains a large subset of a tranlate of a conjugacy class. 

\begin{lem}\label{lem: A is almost contained in a conjugacy class}
For every $\epsilon>0$ there exist $\delta>0$ and $r>0$, such that the following holds. Let $G$ be a finite simple group of rank $\ge r$,
and let $A,B\subseteq G$ be with $|A^{\sigma}B|\le |A|^{\delta}|B|$ for all $\sigma$ in $G$ and with $|B|^{1-\delta} \le |A| \le |B|^{1+\delta}$. Then there exist $a\in A,$ and a conjugacy class $\alpha$ with $|\alpha|\le |A|^{1 + \epsilon}$ such that  
    $|a^{-1}A \cap \alpha| \ge |A|^{1 - \epsilon}.$
\end{lem}
\begin{proof}
We choose $\delta =\delta (\epsilon)$ sufficiently small and $r=r(\delta)$ sufficiently large. By Proposition \ref{prop:small-sets-grow-fast} we may assume that $|A|$ is sufficiently large as a function of $\epsilon$. 
By Lemma \ref{lem:second moment simplified} with $\Gamma = |A|^{1-\delta}$ there exist $a\in A,b\in B$ with 
 \[\frac{|a^{-1}A \cap \alpha| |b^{-1}B \cap \alpha^{-1}|}{|\alpha|} \ge \frac{|A|^{1- \delta} - 3|A|^{\delta}|B|^{\delta}}{|A|^{\delta}|B|^{\delta}} \ge |A|^{1-4\delta} -3,\] provided that $\delta$ is sufficiently small and $r$ is sufficiently large. Rearranging and upper bounding $|a^{-1}A\cap \alpha|\le |A|$ and $|b^{-1}B\cap \alpha^{-1}|\le |B|$
 yields that $|\alpha|\le \frac{|A||B|}{|A|^{1-4\delta} - 3}\le |A|^{1 + \epsilon},$ provided that $\delta$ is sufficiently small and $|A|$ is sufficiently large as a function of $\epsilon.$ 
 We also obtain that 
 \[
  |a^{-1}A \cap \alpha|\ge \frac{|a^{-1}A \cap \alpha| |b^{-1}B \cap \alpha^{-1}|}{|\alpha|} \ge |A|^{1-4\delta} -3\ge |A|^{1-\epsilon}, 
 \]
 provided that $\delta$ is sufficiently small and $|A|$ is sufficiently large as functions of $\epsilon$. 
\end{proof}

We now prove a corresponding lemma that is more suitable for the small set regime.
\begin{lem}\label{lem:A is somewhat contained in a conjugacy class}
    For every $\epsilon \in (0,1/2)$ there exists $\delta,r>0,$ such that the following holds. Let $S$ be a finite simple group of rank $\ge r$, and suppose that $A,B\subseteq S$ are with $|B|^{\epsilon} \le |A|$  and  $|A^{\sigma}B| \le |A|^{1.5}|B|$ for all $\sigma \in S$. Then there exists a conjugacy class $\alpha$ with $|\alpha|\le |A| |B|$, and $a\in A$, such that $|a^{-1}A \cap \alpha|\ge |A|^{0.5-\epsilon}.$
\end{lem}
\begin{proof}
    We choose $\delta =\delta (\epsilon)$ sufficiently small and $r=r(\delta)$ sufficiently large. By Proposition \ref{prop:small-sets-grow-fast} we may again assume that $|A|$ is sufficiently large as a function of $\epsilon$. 
By Lemma \ref{lem:second moment simplified} with $\Gamma = |A|^{1/2}$ there exist $a\in A,b\in B$ with 
 \[\frac{|a^{-1}A \cap \alpha| |b^{-1}B \cap \alpha^{-1}|}{|\alpha|} \ge \frac{|A|^{1/2}}{|A|^{\delta}|B|^{\delta}}-3 \ge |A|^{1/2-
 \delta}|B|^{-\delta} - 3 \ge |A|^{1/2 -\epsilon},\] provided that $\delta$ is sufficiently small and $r,|A|$ are sufficiently large. Rearranging and upper bounding $|a^{-1}A\cap \alpha|\le |A|$ and $|b^{-1}B\cap \alpha^{-1}|\le |B|$
 yields that $|\alpha|\le \frac{|A||B|}{|A|^{1/2-\epsilon}} \le |A||B|,$ provided that $\delta$ is sufficiently small and $|A|$ is sufficiently large as a function of $\epsilon.$ 
 We also obtain that 
 \[
  |a^{-1}A \cap \alpha|\ge \frac{|a^{-1}A \cap \alpha| |b^{-1}B \cap \alpha^{-1}|}{|\alpha|} \ge |A|^{1/2-\epsilon}. 
 \]
\end{proof}

\subsection{Proving Theorems \ref{thm:Skew-Product-Theorem-two-sets} and \ref{thm:Skew-Product-Theorem-effective}}
 We now give a simple lemma showing that lower bounds for $|B\alpha|$ can be automatically translated into lower bounds for $BA,$ if $A$ is sufficiently dense inside $\alpha$. This will allow us to combine Propositions \ref{prop:product.with.normal.set.grows} and \ref{prop:product.with.normal.set.grows.fast} with Lemmas \ref{lem: A is almost contained in a conjugacy class} and \ref{lem:A is somewhat contained in a conjugacy class} to complete the proof.  

\begin{lem}
  \label{lem:bound-with-product-of-subset-and-class}
  Let $\alpha$ be conjugacy class in a finite group $G$.
  For subsets $A,B\subseteq G$
  we have
  $$
  \mathbb{E}_{\sigma \sim G}\big(|A^{\sigma }B|\big) \ge
  |\alpha B| \frac{|A\cap\alpha|}{|\alpha|}
  $$
\end{lem}
\begin{proof}
  We may replace $A$ with $A\cap\alpha$, so we assume that $A\subseteq\alpha$.
  Choose $\sigma \sim G$ randomly. For each element $ab\in \alpha B$
  with $a\in \alpha$ and $b\in B$
  there is a probability of at least $\frac{|A|}{|\alpha|}$ 
  that $a\in A^{\sigma}$. Indeed, $a\in A^{\sigma}$ if and only if $a^{\sigma^{-1}}\in A$ and $a^{\sigma^{-1}}$ is uniformly distributed in the conjugacy class $\alpha$.
  The lemma now follows by taking expectations over $\sigma$ as
  \[\mathbb{E}_{\sigma\sim G}[|A^{\sigma}B|] = \sum_{x\in \alpha B}\Pr_{\sigma\sim G}[x\in A^{\sigma} B] \ge |\alpha B| \frac{|A\cap\alpha|}{|\alpha|}.\] 
\end{proof}

\begin{proof}[Proof of Theorem \ref{thm:Skew-Product-Theorem-two-sets}]
    Let $A,B\subseteq S$ be with $|A|,|B|\le |S|^{1-\epsilon}$ . Let $\epsilon'$ be sufficiently small as a function of $\epsilon$ and $\delta$
    sufficiently small as a function of $\epsilon'$. 
    Suppose that $|A|^{1-\delta} \le |B| \le |A|^{1+\delta}$. Then by Lemma \ref{lem: A is almost contained in a conjugacy class} there exists $a\in A$ and a conjugacy class $\alpha$ with $|\alpha| \le |A|^{1 + \epsilon'}$ and $|a^{-1}A\cap \alpha|\ge |A^{1-\epsilon'}|.$ Let $\tilde{A} = a^{-1}A\cap \alpha$. By Lemma \ref{lem:bound-with-product-of-subset-and-class} there exists $\sigma \in G$ with 
    \[|A^{\sigma}B| \ge |\tilde{A}^{\sigma}  B| \ge |\alpha B| \frac{|\tilde{A}|}{|\alpha|} \ge \frac{|\alpha B|}{|A|^{-\epsilon'}}.\]
    Proposition \ref{prop:product.with.normal.set.grows} now shows that $|\alpha B|\ge |A|^{\epsilon' +\delta}|B|,$ provided that $2\epsilon'+\delta$ is smaller than the constant $\delta$ of Proposition \ref{prop:product.with.normal.set.grows}, and therefore 
    \[
    |A^{\sigma}B|\ge |A|^{\delta}|B|.
    \]
\end{proof}

\begin{proof}[Proof of Theorem \ref{thm:Skew-Product-Theorem-effective}]
    Let $A,B\subseteq S$ be subsetets of a finite simple group with $|A|,|B|\le |S|^{\delta}$. By Corollary \ref{cor:orthogonal-conjugate-bounded-rank} in order to prove our statement, namely that $|A^{\sigma}B|\le |A|^{1/2-\epsilon}|B|^{-\epsilon}$, we may assume that the rank $r$ is sufficiently large as a function of $\epsilon$. Let $\delta$
    be sufficiently small as a function of $\epsilon$. If $|B|^{2\epsilon/3}>|A|,$ then $A^{1.5-\epsilon}|B|^{1-\epsilon}\le |B|$ and the statement holds trivially, so suppose that $|B|^{2\epsilon/3}<|A|.$
    
    By Lemma \ref{lem:A is somewhat contained in a conjugacy class}, provided  that $r$ is sufficiently large, there exists $a\in A$ and a conjugacy class $\alpha$ with $|\alpha| \le |A||B|$ and $|a^{-1}A\cap \alpha|\ge |A^{1/2 - \epsilon/2}|.$ Let $\tilde{A} = a^{-1}A\cap \alpha$. By Lemma \ref{lem:bound-with-product-of-subset-and-class} there exists $\sigma \in G$ with 
    \[|A^{\sigma}B| \ge |\tilde{A}^{\sigma} B| \ge |\alpha B| \frac{|\tilde{A}|}{|\alpha|} \ge \frac{|\alpha B||\tilde{A}|}{|\alpha|}.\]
    Proposition \ref{prop:product.with.normal.set.grows} now implies that $|\alpha B|\ge |\alpha|^{1- \epsilon/2}|B|\ge \alpha |B|^{1-\epsilon/2}|A|^{-\epsilon/2},$ provided that $\delta$ is sufficiently small as a function of $\epsilon$. Therefore 
    \[
    |A^{\sigma}B| \ge |B|^{1-\epsilon/2}|\tilde{A}||A|^{-\epsilon/2}\ge |B|^{1-\epsilon/2}|A|^{1/2 - \epsilon}.
    \]
    This completes the proof of the theorem.
\end{proof}

\section{Proof of \fref{thm:Filling-with-very-large-sets}}
\label{sec:Filling-very-large-Alternating}
\global\long\def\connected{\text{highly connected}}%
\global\long\def\f{\mathcal{F}}%
\global\long\def\E{mathcal{E}}%
\global\long\def\a{\mathcal{A}}%
\global\long\def\pn{\mathcal{P}\left(\left[n\right]\right)}%
\global\long\def\g{\mathcal{G}}%
\global\long\def\Hom{\mathrm{Hom}}%
\global\long\def\l{\mathcal{L}}%
\global\long\def\s{\mathcal{S}}%
\global\long\def\j{\mathcal{J}}%
\global\long\def\d{\mathcal{D}}%
\global\long\def\Cay{\mathrm{Cay}}%
\global\long\def\OPT{\mathrm{OPT}}
\global\long\def\Image{\mathrm{Im}}%
\global\long\def\supp{\mathrm{supp}} 
\global\long\def\GL{\mathrm{GL}}%
\global\long\def\Inf{}%
\global\long\def\Id{\textrm{Id}}%
\global\long\def\Tr{\mathrm{Tr}}%
\global\long\def\sgn{\textrm{sgn}}%
\global\long\def\p{\mathcal{P}}%
\global\long\def\h{\mathcal{H}}%
\global\long\def\n{\mathbb{N}}%
\global\long\def\a{\mathcal{A}}%
\global\long\def\b{\mathcal{B}}%
\global\long\def\c{\mathcal{C}}%
\global\long\def\e{\mathbb{E}}%
\global\long\def\x{\mathbf{x}}%
\global\long\def\y{\mathbf{y}}%
\global\long\def\z{\mathbf{z}}%
\global\long\def\c{\mathcal{C}}%
\global\long\def\av{\mathsf{A}}%
\global\long\def\chop{\mathrm{Chop}}%
\global\long\def\stab{\mathrm{Stab}}%
\global\long\def\Span{\mathrm{Span}}%
\global\long\def\Domain{\mathrm{Domain}}%
\global\long\def\codim{\mathrm{codim}}%
\global\long\def\Var{\mathrm{Var}}%
\global\long\def\rank{\mathrm{rank}}%
\global\long\def\t{\mathsf{T}}%

\newcommand\F[1]{\mathbb{F}_#1}
\newcommand\Sl[2]{\mathrm{SL}_{#1}(\F{#2})}
\newcommand\SLnq{\Sl{n}{q}}
\newcommand{\N}{\mathbb{N}}
\newcommand{\C}{\mathbb{C}}
\newcommand{\bE}{\mathbb{E}}

\newcommand{\set}[1]{\left\{ #1 \right\}}
\newcommand{\eqdef}{\stackrel{\text{def}}{=}}
\newcommand{\sbinom}[2]{\genfrac{[}{]}{0pt}{}{#1}{#2}}
 
\global\long\def\sqbinom#1#2{\left[\begin{array}{c} #1\\ #2 \end{array}\right]}%

\subsection{Notations}
\subsubsection*{Functions on groups}
We shall write $x\sim A$ to denote that $x$ is chosen uniformly out of $A$. For a function on a finite group \(f\in \mathbb{C}[G]\) we therfore write  $\mathbb{E}_{x\sim A}f$ for $\frac{1}{|A|}\sum_{x\in A}f(x)$. We also write $\mathbb{E}[f]$ as a shorthand for $\mathbb{E}_{x\sim G}[f(x)]$. We write $f^{\sigma}$ for the outcome of the right action of $\sigma \in G$ on $f\in \mathbb{C}[G]$, i.e. 
  $f^{\sigma}(\tau ) = f(\tau\sigma)$. The space of functions on $G$ is equipped with the normalized inner product  \[\langle f, g\rangle = \frac{1}{|G|}\sum_{\sigma}f(\sigma)\overline{g(\sigma)}.\]
  We write $\|f\|_p$ for the $L_p$-norm of $f$ given by $\|f\|_p^p = \mathbb{E}[|f|^p]$ and write $\|f\|_{\infty}$ for the maximal value of $|f|$. We shall make use of the following easy special case of Young's convolution's inequality. Namely, that for all functions $f,g$ on a group $G$ we have  \[\|f*g\|_{\infty}\le \|f\|_2\|g\|_2.\]

  \subsubsection*{Probaility measures and density functions}
  Every probability measure $\mu$ on a finite group corresponds to a function $f\colon G\to [0,\infty)$ with $\|f\|_1 = 1$ given by $f(\sigma) = |G|\mu(\sigma).$ The function $f$ is the Radon--Nikodym derivative of $\mu$ with respect to the uniform measure. We call $f\colon G\to [0,\infty)$  with $\|f\|_1 = 1$ a \emph{density function}. Density functions are in one to one correspondence with probability measures. If $f$ is the density function corresponding to $\mu$, then we have 
  \[
  \langle g,f \rangle = \mathbb{E}_{x\sim \mu}[g(x)]. 
  \]

\subsubsection*{Convolutions}
The \emph{convolution of functions} $f$ and $g$ is given by 
\[f*g(x) =\mathbb{E}_{y\sim G}[f(y)g(y^{-1}x)].\]
The \emph{convolution of probability measures} $\mu, \nu$ is the probability measure $\mu *\nu$ corresponding to choosing $x\sim \mu, y\sim \nu$ independently and outputting $xy$. It is easy to see that the density function corresponding to $\mu*\nu$ is the convolution of the correponding density functions.

\subsubsection*{Nonabelian Fourier analysis}

Let $(V,\rho)$ be a complex representation of a finite group $G$, we denote by $V^*$ the dual space of functionals on $V$. Suppose that $\chi = \mathrm{tr} \circ \rho$ is the irreducible character corresponding to $\rho$. Then each function of the form $\sigma \mapsto \varphi(\sigma v)$ with $v\in V$ and $\varphi\in V^*$ is called a \emph{matrix coefficient} for $\chi$. 

We write $\hat{G}$ for the set of irreducible characters of $G$. For $\chi\in \hat{G}$, the space $W_{\chi}$ spanned by the matrix coefficients of $\chi$ is called the \emph{space of matrix coefficients} for $\chi$. 
The space of functions of $G$ is orthogonally decomposed as a direct sum of the spaces $W_{\chi}.$ 
Thus, every function $f$ can be uniquely othogonally decomposed as \[f=\sum_{\chi \in \hat{G}} f^{=\chi},\] with $f^{=\chi}\in W_{\chi}.$  

It is a direct consequence of the Peter--Weyl theorem that the functions $f^{=\chi}$ have the following nice formula. 
\begin{fact}\label{fact:peter--weyl}
    For a function \(f\) on \(G\) we have
    \[f^{=\chi}  = \chi(1)f*\chi.\]
\end{fact}

When $f$ is a class function, we have $f^{=\chi} = \langle f ,\chi \rangle \chi,$ and we write $\hat{f}(\chi) = \langle f ,\chi \rangle.$

\subsection{Our generalized Frobenius formula}

\subsubsection*{Historical flattening lemmas}

 Flattening lemmas concern showing that the $L_2$-norm of $f*f$ is significantly smaller than the $L_2$-norm of $f$.

The following equality about the convolution of two complex valued functions on a finite group was essentially proved by Frobenius.
\[
f*g = \sum_{\chi \in \hat{G}} \frac{\hat{f}(\chi)\hat{g}(\chi)}{\chi(1)}\chi. 
\]

The orthonormality of the characters, then immediately implies the following as a consequence:
\begin{equation}\label{eq:Frobenius}
\|f_1* \cdots * f_m\|_2^2 = \sum_{\chi \in \hat{G}} \chi(1)^{2-2m}\prod_{i=1}^{m}|\hat{f_i}(\chi)|^{2}.
\end{equation}

The equality \eqref{eq:Frobenius} has been used extensively in the study of mixing times in groups (see e.g. \cite{diaconis1988group} for a survey).  

In the context of the spectral theory of automorphic forms, Sarnak and Xue \cite{sarnak1991bounds} were interested in similar inequalities that hold also when $f,g$ are not necessarily class functions. They then came up with the following inequality that holds for arbitrary functions $f\in \mathbb{C}[G]$. 
\[
\|f*g\|_2^2 \le \sum_{\chi \in \hat{G}} \frac{\|f^{=\chi}\|_2^2\|g^{=\chi}\|_2^2}{\chi(1)}.
\]

We will be mostly interested in inequalities of this type when $f$ is a density function. The inequlity should be then compared to the following consequence of Young's convolution inequality 
\[
\|f*g\|_2^2 \le \|f\|_1^2 \|g\|_2^2 = \sum_{\chi} \|g^{=\chi}\|_2^2.
\]

Therefore, when one has an upper bound of the form $\|f^{=\chi}\|_2 \le \epsilon \sqrt{\chi(1)}$ for all irreducible characters $\chi$, then they can deduce from the Sarnak--Xue inequality the better flattening result \( \|f*g\|_2 \le \epsilon\|g\|_2.\) This was exactly the approach carried out in Keevash and Lifshitz~\cite{keevash2023sharp}, where the inequalities of the form \(\|f*g\|_2 \le \epsilon\|g\|_2\), were obtained via hypercontractivity.

\subsubsection*{Our flattening lemma}

In this section our main goal is to deal with the normalized indicators $f= \frac{1_A}{\mu(A)}$ of sets $A\subseteq G$ with $|A|>|G|^{1-\epsilon}.$   For such functions inequalities of the form $\| f^{=\chi} \|_2 \le \epsilon \sqrt{\chi(1)}$ are out of reach of the available hypercontractive technology. Instead, we prove a new flattening lemma, which is applicable under the weaker condition \[\|f^{=\chi}\|_2 \le \chi(1)^{0.9}.\]

\begin{defn}
For a function $f\in \mathbb{C}[S_n]$ we write $f^{\sigma}(\tau) = f(\tau\sigma)$.
\end{defn}

We are now ready to state the following generalized Frobenius formula, which is a generalization of \eqref{eq:Frobenius} when the functions are not necessarily class functions.

\begin{lem}\label{lem:Frobenius}
Let $G$ be a finite group. Then 
\[
\mathbb{E}_{\sigma_1,\ldots, \sigma_m\sim G}\|f_1^{\sigma_1}* f_2^{\sigma_2} * \cdots *f_m^{\sigma_m} * f_{m+1}\|_2^2 = \sum_{\chi \in \hat{G}} \frac{\prod_{i=1}^{m+1} \|f_i^{=\chi}\|_2^2}{\chi(1)^{2m-2}}.
\]
\end{lem}

We prove the lemma via a repeated application of the following claim. Recall that a unitary representation of a group $G$ is a Hilbert space $H$, such that $\langle gv, gu\rangle = \langle v,u\rangle$ for all $v,u\in H.$
For $v\in H$ we write 
\[f*v = \mathbb{E}_{g\sim G}[f(g)g v].\]
In the special case, where $H$ is the left regular action, which is $H = \mathbb{C}[G]$ with the action $\sigma(f)(\tau) = f(\sigma^{-1}\tau),$ the two notions of convolution agree. 

We also recall that if $\chi = \mathrm{tr} \circ \rho$ is an irreducible character, then the $\chi$-isotypic component of $H$ is the sum of all subrepresentations of $H$ that are isomorphic to $\rho$. We write $H_{\chi}$ for the $\chi$-isotypic component and $v^{=\chi}$ for the orthogonal projection of $v$ onto $H_{\chi}.$ Recall also that the space $W_{\chi}$ is isomorphic to $\chi(1)$ copies of the representation $\overline{\chi},$ and that $\overline{\chi}$ is the trace of the dual representation to $\chi$ given by $g\mapsto \rho(g^{-1})^{t}.$

\begin{claim}\label{claim: Frobenius}
    Let $H$ be a unitary representation of a finite group $G$, and let $v\in H$. Then 
    \[\mathbb{E}_{\sigma \sim G}\|f^{\sigma} * v\|^2 = \sum_{\chi}\frac{1}{\chi(1)^2} \|f^{= \overline{\chi} }\|_2^2\|v^{=\chi}\|^2.\]
\end{claim}
\begin{proof}
    We may decompose $H$ orthogonally as a direct sum of irreducible representations ${V_i}$. As for each $v\in H$ convolutions of the form $f*v$ are linear combinations of translations of $v$, we have $f*v \in V_i$ for all $v\in V_i$ and $f\in \mathbb{C}[G].$
    
    Decomposing orthogonally $v = \sum v_i$ with $v_i\in V_i$, then $f^{\sigma} * v = \sum f^{\sigma}*v_i$, and we obtain that $f^{\sigma}*v_i\in V_i$ for all $i$. Taking $L_2$-norms, we obtain that the claim would follow once we prove it in the special case where $H$ is an irreducible representation.
    
    Let us assume that $H$ is an irreducible representation corresponding to a character $\chi$, and let $e_1,\ldots, e_d$ be an orthonormal basis for $H$, with $d=\chi(1)$ being the dimension of $H$. Let $e_{ij}$ be the functions on $G$ given by $e_{ij}(g) = \langle  e_i ,  g e_j \rangle$. The Peter--Weyl theorem now states that following facts. First, the vectors $\sqrt{\dim(H)}e_{ij}$ constitute an orthonormal basis for the space $W_{\overline {\chi}}$ of matrix coefficients for the dual representation. Moreover, $f*v = f^{=\overline{\chi}} * v$ for all $f\in \mathbb{C}[G]$ and $v\in H$. Finally, we have $e_{ij} * e_k = \frac{1}{\dim(H)} e_i \delta_{jk}$.

    Let us write $f= \sum_{i,j} a_{ij} e_{ij}$. Let us also write $\varphi_i$ for the functional whose values on the basis elements is given by $e_j \mapsto \frac{1}{\dim(H)} a_{ij}$ for all $j$. 
    
    Then $f*u = \sum_{i=1}^{d} e_i \varphi_i(u)$ for all $u$, as this holds for the basis elements. Thus, 
    \[\|f^{\sigma} * v\|^2 = \|f * (\sigma v)\|^2 = \sum_{i=1}^{d} \|\varphi_i (\sigma v)\|^2.\]
    Now by the Peter--Weyl theorem, the square of the $L^2$-norm of the matrix coefficient  $g\mapsto \varphi_i g v$ is given by $\sum \frac{a_{ik}^2}{\dim(H)^3} \|v\|^2.$ Thus,  
    \[\mathbb{E}_\sigma \|f^{\sigma}*v\|^2 = \sum_{i,j} \frac{a_{ij}^2}{\dim (H)^3} = \frac{\|f\|_2^2}{\dim(H)^2}.\] 
\end{proof}

We are now ready to prove Lemma \ref{lem:Frobenius}.
\begin{proof}[Proof of Lemma \ref{lem:Frobenius}]
    The lemma follows by applying the claim repeatedly with the left regular action of $G$ on $\mathbb{C}[G]$, given by $\sigma (f)(\tau) = f(\sigma^{-1}\tau)$. Here we make use of the fact that the space $W_{\chi}$ of matrix coefficients for $\chi$ is also the isotypic component for the representation $\overline{\chi}$, with respect to the left regular representation.   
\end{proof}

\subsection{Character theoretic conditions for product decompositions}

In this section we prove the following criterion for subsets of a finite simple group that implies that a products of their conjugates is the whole group. Our results holds more generally for sets that have good bounds for the Witten zeta function, given by $\zeta_G(s) = \sum_{\chi \in \hat{G}} \chi(1)^{-s}.$ 

We make use of the following estimate for the Witten zeta function of finite simple groups. It is due to Liebeck and Shalev~\cite{liebeck2005fuchsian}.

\begin{thm}\label{thm:LS}
For each $t>1,\epsilon >0$, there exists $n_0>0$, such that if $G$ is a finite simple group, with $|G|>n_0$, then $\sum_{\chi\in \hat{G}\setminus\{1\}} \chi(1)^{-t} < \epsilon.$
\end{thm}

We are now ready to prove our character theoretic criterion for product decompositions.

\begin{thm}\label{thm:character theoretic criterion for product decompositions}
    Let $t,\alpha>0,$ and $m$ be an even integer with $m\ge \frac{2 + 2t}{\alpha}$. Suppose that $G$ is a finite group with $\zeta_G (t) < 2^{-m/2 - 1}$, and let $A_1,\ldots, A_m\subseteq G$ be, such that 
    for all irreducible representation $\chi$ and all $i$ we have
    \[
    \mathbb{E}_{a,b\sim A_i}[\chi(a^{-1}b)]<2 \chi(1)^{1-\alpha}.
    \]
    Then there exists $\sigma_1,\ldots, \sigma_m$, such that $A_1^{\sigma_1}\cdots A_m^{\sigma_m} = G.$
    In particular, this holds for finite simple groups $G$ with $t= 1.1$, provided that $|G|\ge n_0=n_0(m)$.
\end{thm}
\begin{proof}
    Let $f_i = \frac{1_{A_i}}{\mu(A_i)}.$  
    We show something stronger, namely, that there exists $\sigma_1,\ldots, \sigma_m$ with
    \begin{equation}\label{eq:l-infinity mixing time}
        \|f_1^{\sigma_1} \cdots f_m^{\sigma_m}  - 1\|_{\infty} < 1/2.
    \end{equation} 
    Proving \eqref{eq:l-infinity mixing time} would indeed complete the proof as $f_{i}^{\sigma}$ is the density function for the uniform measure on the set $A_i\sigma_i^{-1}.$ Therefore, the function $f_1^{\sigma_1} * \cdots * f_m^{\sigma_m}$ is supported on $A_1\sigma_1^{-1}A_2\sigma_2^{-1}\cdots A_m\sigma_m^{-1}$. On the other hand, this function takes a value at least $1/2$ on all of $G$. Hence, once we prove \eqref{eq:l-infinity mixing time} we will be able to deduce that $A_1\sigma_1^{-1}\cdots A_m \sigma_m^{-1} = G.$ Since there are translates of the sets $A_i$ whose product is $G$, there also exist conjugates of the sets $A_i$ whose product is $G$.

    We now establish \eqref{eq:l-infinity mixing time} for some $\sigma_i$.  
    We make use of the fact that $f*g = (f-\mathbb{E}[f]) * (g - \mathbb{E}[g]) + \mathbb{E}[f]\mathbb{E}[g]$ for all two functions $f$ and $g$. We will also make use of the easy special case of Young's convolution inequality implying that $\|f*g\|_{\infty} \le \|f\|_2\|g\|_2.$ 
    These facts together yield that 
    \[
        \|f_1^{\sigma_1} \cdots f_m^{\sigma_m} - 1\|_{\infty}\le \left\|(f_1^{\sigma_1} - 1)\cdots (f_{m/2}^{\sigma_{m/2}}-1)\right\|_2 \left\|(f_{m/2 + 1}^{\sigma_{m/2 +1}} -1) \cdots (f_{m}^{\sigma_{m}} - 1)\right\|_2.
    \]
    
    We may now take expectations over $\sigma_1,\ldots, \sigma_m$ and apply Cauchy--Schwarz to obtain that
    \begin{align*}
     \mathbb{E}_{\sigma_1,\ldots, \sigma_m}\left[ \left\|(f_1^{\sigma_1} - 1)\cdots (f_{m/2}^{\sigma_{m/2}}-1) \right\|_2 \left\|(f_{m/2 + 1}^{\sigma_{m/2 +1}} -1) \cdots (f_{m}^{\sigma_{m}} - 1)\right\|_2\right]  \le & \\ 
     \sqrt{\mathbb{E}_{\sigma_1,\ldots, \sigma_{m/2}}\left[ \left\|(f_1^{\sigma_1} - 1)\cdots (f_{m/2}^{\sigma_{m/2}}-1) \right\|_2^2 \mathbb{E}_{\sigma_{m/2+1}, \ldots, \sigma_m} \left\|(f_{m/2 + 1}^{\sigma_{m/2 +1}} -1) \cdots (f_{m}^{\sigma_{m}} - 1)\right\|_2^2\right]}
    \end{align*}
    By Lemma \ref{lem:Frobenius} (in conjunction with the fact that translation does not affect the value of $\|f^{=\chi}\|_2$), the right hand side is equal to 
    \[
    \sum_{\chi \in \hat{G}\setminus{1}}\chi(1)^{1-m}\prod_{i=1}^{m}\|f_i^{=\chi}\|_2.
    \]
    Now $f^{=\chi} = \chi(1) f*\chi$ and as \begin{equation}\label{eq:convolve}
    \|f^{=\chi}\|_2^2 = \chi(1) \langle f *\chi, f\rangle = \chi(1)\mathbb{E}_{b\sim A}[f*\chi(b)] = \chi(1)\underset{b,a\sim A}{\mathbb{E}}\chi(a^{-1}b), 
    \end{equation}
    we have 
    \[
    \|f^{=\chi}\|_2^2 \le 2\chi(1)^{2-\alpha}.
    \]
    Therefore, we may combine our inequalities to obtain that  
    \begin{align*}
        \mathbb{E}_{\sigma_1,\ldots, \sigma_m} \| f_1^{\sigma_1} * \cdots * f_{m}^{\sigma_m} -1 \|_{\infty} \le \sum_{\chi \in \hat{G}\setminus \{1\} } 2^{m/2}\chi(1)^{1 - m\alpha/2} & \\ \le \sum_{\chi \in \hat{G}\setminus \{1\}} 2^{m/2} \chi(1)^{-t} \le 1/2. 
    \end{align*}
    This completes the proof of \eqref{eq:l-infinity mixing time}. The `in particular' part now follows from Theorem \ref{thm:LS}.
\end{proof}

\begin{proof}[Proof of Theorem \ref{thm:character theoretic criterion for product decompositions in finite simple groups}]
    The statement follows immediately from Theorem \ref{thm:character theoretic criterion for product decompositions}.
\end{proof}

\subsection{Proof of Theorem \ref{thm:Filling-with-very-large-sets} for groups of Lie type}
To prove  Theorem \ref{thm:Filling-with-very-large-sets} for groups of Lie type we make use of the following theorem of Guralnick, Larsen, and Tiep~\cite{guralnick2020character,guralnick2024character}. 
\begin{thm}\label{thm:GLT}
    For every $\epsilon >0,$ there exists $\delta,r>0$, such that if $G$ is a finite simple group of Lie type whose rank is $\ge r$, then there exists a set $A$, with $|A|<|G|^{1-\delta},$ such that for all $\sigma \in G\setminus A$, and all $\chi \in \hat{G}$ we have 
    \[|\chi(\sigma)|\le \chi(1)^{\epsilon}.\] 
\end{thm}

\begin{thm}\label{thm: filling large lie types}
There exis $c>0,n_0>0$ such that if $G$ is a finite simple group of Lie type with $|G|>n_0$ and $A_1,\ldots, A_6\subseteq G$ have size $\ge |G|^{1-c},$ then there exist $\sigma_1,\ldots, \sigma_6$ with $A_1^{\sigma_1}\cdots A_{6}^{\sigma_6} = G.$
\end{thm}
\begin{proof}
We may assume that the rank of $G$ is sufficiently large for otherwise the theorem follows from the main result of Nikolov and Pyber~\cite{nikolov2011product}.
We may therefore apply Theorem~\ref{thm:GLT} to show that there exists a constant $c>0$ and a set $B$ of size $< |G|^{1-100 c},$ such that for all $\sigma \in G\setminus B$ we have $|\chi(\sigma)|\le \chi(1)^{0.1}.$
    Let $A_1,\ldots, A_6$ have size $>|G|^{1-c}$. Let $f_i = \frac{1_{A_i}}{\mu(A_i)}.$ Then by \eqref{eq:convolve} for all $\chi$ with $\chi(1)>|G|^{c}$ we have $\mathbb{E}_{a,b\sim A_i}[ \chi[a^{-1} b] =  \|f_i^{=\chi}\|_2^2/\chi(1) \le \|f_i\|_2^2/\chi(1) \le |G|^{c}/\chi(1)] \le 1\le \chi(1)^{0.1}.$
    
    Moreover, for characters $\chi$ with $\chi(1)<|G|^{c}$
    and all $i$ we have \[|\mathbb{E}_{a,b \sim A_i} \chi(a^{-1}b)| \le \mathbb{E}_{a,b\sim A_i} [1_{G\setminus B} \chi(1)^{0.1} + 1_{B}(a^{-1}b) \chi(1)]\] 
    Now for each choice of $a$ there are at most $|B|$ choices of $b$ with $a^{-1}b\in B$ and therefore
    \[\Pr_{a,b\sim A_i} [a^{-1}b\in B ]\le \frac{|B|}{|A_i|} \le |G|^{-99c} \le \chi(1)^{-99}.\]
    Hence, 
    \[|\mathbb{E}_{a,b \sim A_i} \chi(a^{-1}b)| \le \chi(1)^{0.1}+\chi(1)^{-98}\le  2\chi(1)^{0.1}\] for each $i$. Theorem \ref{thm:character theoretic criterion for product decompositions}
    now completes the proof of the statement.
\end{proof}

\subsection{Proof overview for the case proof of
alternating groups}

For groups of Lie type $G$ we proved Theorem~\ref{thm:Filling-with-very-large-sets} by making use of Theorem \ref{thm:character theoretic criterion for product decompositions}, which reduced the product decomposition statement to proving that inequalities of the form
\begin{equation}\label{eq: Character theoretic goal}
|\mathbb{E}_{a,b\sim A}[\chi(a^{-1}b)]| \le \chi(1)^{1-c'}
\end{equation}
hold for every set $A$ of size $>|G|^{1-c}$ for sufficiently small absolute constants $c,c'>0$. For alternating groups, a correponding statement fails even for conjugacy classes of size $n!^{0.99}$. Indeed, it can be eaily shown that \eqref{eq: Character theoretic goal} would imply constant mixing time for the random walk on the Cayley graph $\mathrm{Cay}(G,A)$. On the other hand, such a conjugacy class $A$ can have $\Theta(n)$ fixed points and then the corresponding random walk  has $\Theta(\log n)$ mixing time.

We circumvent this by proving structural reults for the sets $A\subseteq A_n$ for which \eqref{eq: Character theoretic goal} fails, and then leveraging the structure to complete the proof. For a set $I$ of size $d$ and $\sigma \in A_n$ we call the set of permutations in $A_n$ that agree with $\sigma$ on the set $I$ a $d$\emph{-umvirate}, and denote it by $U_{I,\sigma}\}$. We then say that a set $A$ is $r$-\emph{global} if $\frac{|A\cap U|}{|U|}\le r^{d}\mu(A)$ for every $d$-umvirate $A$. 

We prove that \eqref{eq: Character theoretic goal} for sufficiently global sets $A$.

\subsection{Larsen--Shalev type character bounds}
We now move on to proving character bounds that are suitable for proving Theorem \ref{thm:Filling-with-very-large-sets} for alternating groups. 

The following bound shows that characters of $S_n$ takes large values only on permutations with various fixed points. It follows from Larsen and Shalev~\cite{larsen2008characters}[Theorem 1.3] 

\begin{thm}\label{thm:Larsen shalev 1.3}
    For every $\epsilon>0$ there exists $n_0$, such that if $n>n_0$ then for every character $\chi$ of $S_n$ and a permutation $\sigma$ with $\le n^{1-\epsilon}$ fixed points we have \[
    |\chi(\sigma)|\le \chi(1)^{1-\epsilon/3}.
    \]
\end{thm}

In this section we prove a variant of the Larsen--Shalev character bounds that is more suitable for our purposes. It also shows that characters take large values only on permutations with various fixed points, but it is stronger for characters corresponding to high dimensional representations.

We follow the proof strategy of Larsen and Shalev and only deviate at the final point. The key idea in their was to give the following definition of virtual degree. In the below we use the convention that $i!$ is 1 whenever $i$ is a nonpositive integer. 

\begin{defn}
Let $\lambda$ be a partition and $\lambda'$ be its conjugate. Then the \emph{virtual degree} is given by:
\[D(\lambda) = \frac{(n-1)!}{\prod (\lambda_i - i) \cdot \prod (\lambda'_i-i).}\]
\end{defn}

Following Larsen--Shalev we write $d(\lambda)$ for the dimension of the Specht module $V_{\lambda}$. They proved that the dimension $d(\lambda)$ and the virtual degree $D(\lambda)$ are not too far off of one another. 

\begin{lem}[Larsen--Shalev {\cite{larsen2008characters}}]\label{lem:asymptotics for virtual degree}
    $D(\lambda)\le d(\lambda) \le D(\lambda)^{1+o(1)}$.
\end{lem}
We say that $\tilde{\lambda}$ is a subdiagram of $\lambda$ if its Young diagram is obtained from the Young diagram of $\lambda$ by deleting boxes without touching the other boxes. We write $|\lambda|$ for the number of boxes of $\lambda$. The following lemma of Larsen and Shalev bounds the virtual degree of a subdiagram in terms of the virtual degree of the original diagram.  

\begin{lem}[{\cite{larsen2008characters}}] \label{lem:bounds on virtual degrees of subdiagrms}
Let $\tilde{\lambda}$ be a subdiagram of $\lambda.$ Then   
    $\frac{\log D(\tilde{\lambda})}{\log D(\lambda)}\le \frac{\log |\tilde{\lambda|}}{\log |\lambda|}$.
\end{lem}

Finally we need the following bound of Larsen and Shalev, which they deduced by repeatedly applying the Murnaghan--Nakayama rule. 
\begin{lem}[{\cite{larsen2008characters}}]\label{lem:removing all k-cycles}
Let $\sigma = 1^{c_1}2^{c_2}\cdots k^{c_k}$ be a permutation in $S_n.$ Let $\sigma' = 1^{c_1}2^{c_2}\cdots (k-1)^{c_{k-1}}$ be the permutation obtained by removing all the $k$-cycles of $\sigma$, and set $n'$ to be with $\sigma'\in S_{n'}$.  Then  
    \[
    \chi_{\lambda}(\sigma) \le \max_{\tilde{\lambda}} \left(\frac{D(\lambda)}{D(\tilde{\lambda})}\right)^{1/k}\chi_{\tilde{\lambda}}(\sigma'),
    \]
    where the maximum is taken over all subdiagrams of $\lambda$ with $n'$ boxes. 
\end{lem}

We give the following variant of the Larsen--Shalev bound. For a given permutation $\sigma$, write $\Sigma_i= \Sigma_i(\sigma)$ for the number of elements of $[n]$ that appear inside a $j$-cycle of $\sigma$ for $j\le i$. $e_i = e_i(\sigma)$ are given by \[\log_n |\Sigma_i| - \log_n|\Sigma_{i-1}|\] 

Larsen and Shalev proved that \[\chi_{\lambda}(\sigma) \le D(\lambda)^{{ \sum_{i=1}^{k}\frac{e_i}{i} + o(1)}} = d(\lambda)^{{ \sum_{i=1}^{k}\frac{e_i}{i} + o(1)}}.\]
We now give the following variant of their theorem that is more suitable for our purposes.

\begin{thm}\label{thm:Larsen--Shalev variant}
Let $\sigma$ be a permutation with $f$ fixed points and write $e_i = e_i(\sigma)$. Then 
    \[
    \chi(\sigma)\le D(\lambda)^{e_1/2 + \sum_{i=2}^{k}\frac{e_i}{i}} f!^{1/4}
    \]
\end{thm}
\begin{proof}
We prove the theorem by induction on $k$ simultaneously for all $n$ and $\lambda$. For $k=1$ the inequality becomes 
\[\chi(1) = d(\lambda) \le \sqrt{D(\lambda) n!^{1/4}}.\] Now $d(\lambda) \le D(\lambda)$ by Lemma \ref{lem:asymptotics for virtual degree} and $d(\lambda)\le n!^{1/2}$ by the well known fact that $\sum_{\lambda} d({\lambda})^2 = n!$. The case $k=1$ of the induction follows by taking the geometric mean of the above inequalities. 

We now prove the induction step. By Lemma \ref{lem:removing all k-cycles} we have 
\[\chi_{\lambda}(\sigma) \le \left(\frac{D(\lambda)}{D(\tilde{\lambda})}\right)^{1/k}\chi_{\tilde{\lambda}}(\sigma'),\]
where $\sigma'$ is obtained from $\sigma$ by removing all its $k$-cycles and $\tilde{\lambda}$ is a suitable subdiagram of $\lambda$. 

Write $e_i' = e_i(\sigma')$. Then by induction we have 
\begin{equation}\label{eq:improved larsen shalev}
    \chi_{\lambda}(\sigma)\le \left(\frac{D(\lambda)}{D(\tilde{\lambda})}\right)^{1/k} D(\tilde{\lambda})^{e_1'/2 + \sum_{i=2}^{k}\frac{e_i'}{i}} f!^{1/4}.
\end{equation}
Note that we have $e_i' = e_i \frac{\log n'}{\log n} = \frac{e_i}{1-e_k}.$ We also have $D(\tilde{\lambda}) \le D(\lambda)^{1-e_k}$ by Lemma \ref{lem:bounds on virtual degrees of subdiagrms}. We now assert that the exponent of $D(\tilde{\lambda})$ in the left hand side of \eqref{eq:improved larsen shalev} is non-negative. Indeed, we have \[e_1'/2 + \sum_{i=2}^{k-1} \frac{e_i'}{i} - 1/k \ge \sum_{i=1}^{k-1} e_i'/k -1/k =0.\] Thus, we may replace $D(\tilde{\lambda})$ by $D(\lambda)^{1-e_k}$  in the left hand size of \eqref{eq:improved larsen shalev} and obtain that \[\chi(\sigma) \le D(\lambda)^{\frac{e_k}{k} + e_1/2 + \sum_{i=2}^{k-1} e_i/i}f!^{1/4},\] which completes the proof of the theorem.
\end{proof}

We are now ready to prove the following character bounds for permutations having few fixed points. 

\begin{cor}\label{cor: Larsen--Shalev}
    There exists $c,n_0>0$, such that if $n>n_0$, $\lambda$ is a partition of $n$, and $\sigma$ is a permutation with at most $c\frac{\log \chi_{\lambda}(1)}{\log \log \chi_{\lambda}(1)}$ fixed points, then 
    \[
    |\chi_{\lambda}(\sigma)| \le \chi_{\lambda}(1)^{0.51}.
    \]
\end{cor}
\begin{proof}
    We have $D(\lambda)\le d(\lambda)^{1+o(1)}$ by Lemma \ref{lem:asymptotics for virtual degree}. We also have $e_1/2+\sum_{i=2}^{k}e_i/i\le \sum_{i=1}^{k} e_i/2 = 1/2.$ Let $f$ be the number of fixed points of $\sigma$. Then by Stirling's formula we have $f! \le \chi_{\lambda}(1)^{0.005}$ provided that $c$ is sufficiently small. This shows that 
    \[
    \chi_{\lambda}(\sigma) \le \chi_{\lambda}(1)^{0.505} +o(1),
    \]
    so the corollary follows provided that $n_0$ is sufficiently large. 
\end{proof}

\subsection{Adaptation from symmetric groups to alternating groups}

Recall that a character $\chi$ of $S_n$ either restricts to an irreducible character of $A_n$ or restricts to a sum of two characters $\chi_1, \chi_2$ with $\chi_1(\sigma) = \chi_2(\sigma^{(12)})$ for all $\sigma\in A_n.$ 
In the latter case, we either have $\sigma^{A_n} = \sigma^{S_n}$ and then $\chi_1(\sigma) = \chi_2(\sigma ) = 1/2 \chi(\sigma)$ or we have $\sigma^{A_n}\ne \sigma^{S_n}$. This leaves the case $\sigma^{A_n}\ne \sigma^{S_n}$ as the only case where the character bounds for $A_n$ do not follow immediately from those of $S_n$. 

To handle this case
Larsen and Tiep~\cite{larsen2023squares} proved the following bound.

\begin{thm}\label{thm:Larsen--Tiep}
For every $\epsilon>0$ there exists $n_0>0$, such that the following holds. Suppose that $n>n_0$ and that 
$\sigma^{A_n} \ne \sigma^{S_n}$. Then for every character $\chi$ of $A_n$ we have $|\chi(\sigma)|\le \chi(1)^{\epsilon}.$
\end{thm}

Combining Theorem \ref{thm:Larsen--Tiep} with Theorem \ref{thm:Larsen shalev 1.3} yields the following. 
\begin{cor}\label{cor:alternating Larsen--Shalev}
    For every $\epsilon>0$ there exists $n_0$, such that the following holds. Suppose that $n>n_0$ and that $\sigma\in A_n$ has at most $n^{1-\epsilon}$ fixed points. Then for every character $\chi$ of $A_n$ we have $|\chi(\sigma)|\le \chi(1)^{1-\epsilon/3}.$
\end{cor}
\begin{proof}
    The statement clearly holds if $\chi$ is the restriction of an irreducible character of $S_n$. 
    Otherwise the function $\chi':= \chi(\sigma) + \chi(\sigma^{(12)})$ is an irreducible character of $S_n$. For $\sigma$ with $\sigma^{A_n} = \sigma^{S_n}$ we have $\chi(\sigma) = \chi'(\sigma)/2$ and $\chi(1) = \chi'(1)/2$. Therefore the theorem holds in this case as well. Finally, if $\sigma^{A_n}\ne \sigma^{S_n}$ the statement follows from Theorem \ref{thm:Larsen--Tiep}. 
\end{proof}

A similar proof yields the following variant of Corollary \ref{cor: Larsen--Shalev}.  
\begin{thm}\label{thm:Larsen--Shalev_alternating}
    There exists $c,n_0>0$, such that if $n>n_0$, $\chi$ is a character of $A_n$ and $\sigma\in A_n$ is a permutation with at most $c\frac{\log \chi(1)}{\log \log \chi(1)}$ fixed points, then 
    \[
    |\chi(\sigma)| \le \chi(1)^{0.51}.
    \]
\end{thm}

\subsection{Fourier anti-concentration inequalities for global functions}
Recall that a coset of the pointwise stabilizer of a set of size $d$ is called a $d$-umviarate.

Suppose that we want to show that the product of few conjugates of a set $A\subseteq A_n$ is equal to $A_n$. By Theorem \ref{thm:character theoretic criterion for product decompositions} it is sufficient to show that $|\mathbb{E}_{a,b\sim A}[\chi(a^{-1}b)]|\le 2\chi(1)^{1-c}$ for an absolute constant $c>0$. 
The following lemma shows that we have that unless $A$ has a significant density increment inside a $d$-umvirate.  
\begin{lem}\label{lem: density increment if not concentrated}
    There exists $c>0$, such that for every $\epsilon>0$ there exists $\delta>0$ and an integer $n_0>0$, such that the following holds. Let $n>n_0$, let $A\subseteq A_n$, and let $\chi$ be a character of $A_n$.
    Suppose that $|\mathbb{E}_{a,b\sim A}[\chi(a^{-1}b)| \ge 2 \chi(1)^{1-\delta}$, and write  $d = \left \lfloor\max\left( c \frac{\log \chi(1)}{\log \log \chi(1)}, n^{1-3 \delta}\right)\right \rfloor.$ Then there exists a $d$-umvirate $U$, such that \[ \frac{|A\cap U|}{|U|} \ge n^{(1-\epsilon) d}\mu(A).\]
\end{lem}
\begin{proof}
    We will choose the parameters so $n_0$ is sufficiently large with respect to the rest of the parameters. 
    Let $B\subseteq A_n$ be the set of permutations with $|\chi(\sigma)| > \chi(1)^{1-\delta}.$ Then every permutation in $B$ has more than $n^{1-3\delta}$ fixed points by Theorem \ref{thm:Larsen shalev 1.3}, provided that $n_0$ is sufficiently large with respect to $\delta$. Moreover, it has more than $c \frac{\log \chi(1)}{\log \log \chi(1)}$ fixed points by Theorem \ref{thm:Larsen--Shalev_alternating}, provided that $c$ is a sufficiently small constant and $\delta\le 1/3$. This shows that every permutation in $B$ has at least $d$ fixed points. 
    As $\chi(1) = \|\chi\|_{\infty}$ we also have 
    \[
    2\chi(1)^{1-\delta} \le |\mathbb{E}[\chi (a^{-1}b)]| \le \chi(1)^{1-\delta}  + \Pr_{a,b\sim A}[a^{-1}b\in B]\chi(1). 
    \]
    Thus, $\Pr_{a,b\sim A}[a^{-1}b\in B] \ge \chi(1)^{-\delta}.$
    As every permutation in $B$ has at least $d$ fixed points, a union bound implies that there exists a set $I$ of size $d$, such that for independent random $a,b\sim A$ the probability that $a^{-1}b$ fixes $I$ is at least $\frac{\chi(1)^{-\delta}}{\binom{n}{d}}.$
    Let $H$ be the pointwise stabilizer of the set $I$.  We now obtain that there exists $a_0\in A,$ such that the probability that a random $b\sim A$ belongs to the $d$-umvirate $U := a_0H$ is at least $\frac{\chi(1)^{-\delta}}{\binom{n}{d}}.$
    This shows that $|A\cap U|\ge \frac{\chi(1)^{-\delta}}{\binom{n}{d}}|A|$ and therefore 
    \[
    \frac{|A\cap U|}{|U|} \ge \mu(A)\chi(1)^{-\delta} d! \ge \mu(A) n^{(1-\epsilon)d},
    \]
    provided that $\delta$ is sufficiently small with respect to $c$ and $n_0$ is sufficiently large with respect to the rest of the parameters. Indeed, $d\ge n^{1-3\delta}$ and therefore, provided that $n$ sufficiently large we have $d!\ge d^{(1-\delta)d} \ge  n^{(1-3\delta)(1-\delta) d}$. Moreover, $\chi(1)\le d^{Ccd}\le n^{Ccd},$ for some absolute constant $C>0$ and therefore $\chi(1)^{-\delta} \ge n^{-Cc\delta d}.$ Combining these yields that 
    \[
    \frac{|A\cap U|}{|U|} \ge n^{(1-3\delta)(1-\delta) d - \delta cCd}\mu(A) \ge n^{(1-\epsilon)d},
    \] 
    provided that $\delta$ is sufficiently small with respect to $\epsilon$ and the absolute constants $c,C>0$. 
\end{proof}

Recall that a set $A$ is $r$-global if for each $d$ and each $d$-umvirate $U$ we have $\frac{|A\cap U|}{|U|}\le r^d\mu(A).$ 
\begin{proof}[Proof of Theorem \ref{thm:anti-concentration inequalities for global functions}]
    The theorem follows immediately from Lemma \ref{lem: density increment if not concentrated}.   
\end{proof}

Lemma \ref{lem: density increment if not concentrated} in conjunction with our character theoretic criterion for product decompositions (Theorem \ref{thm:character theoretic criterion for product decompositions}) imply the following product decomposition result for sufficiently global sets. 

\begin{lem}\label{lem:globalness implies covering}
    For each $\epsilon>0$ there exist integer $n_0,M>0$, such that if $A\subseteq A_n$ is $n^{1-\epsilon}$-global, then there are $M$ conjugates of $A$ whose product is $A_n.$
\end{lem}
\begin{proof}
    Let $A$ be $n^{1-\epsilon}$-global.
    By Lemma \ref{lem: density increment if not concentrated} there exists $\delta= \delta(\epsilon)$, such that $|\mathbb{E}_{a,b\sim A}\chi(a^{-1}b)|\le 2\chi(1)^{1-\delta}.$ The lemma now follows from Theorem \ref{thm:character theoretic criterion for product decompositions}.
\end{proof}

\subsection{Proof of Theorem \ref{thm:Filling-with-very-large-sets}}

\begin{lem}\label{lem:reduction to umvirates}
There exists $M>0$, such that for every $\epsilon>0$, there exists $M,n_0>0$, such that if $n>n_0$ and $A\subseteq A_n$ has size $\ge |G|^{1-\epsilon/2},$ then a product of $M$ translates of $A$ containing the pointwise stabilizer of a set $I$ of size $\le \epsilon n.$  
\end{lem}
\begin{proof}
    Let $d\ge 0$ be maximal, such that there exists a $d$-umvirate $U$ with $\frac{|A\cap U|}{|U|}\ge \mu(A)n^{d/2}.$ Then $n^{d/2}\le |G|^{\epsilon/2}\le n^{n\epsilon/2},$ which implies that $d\le \epsilon n.$  Write $U = aH,$ where $H$ is the subgroup of all even permutations fixing a set $I$ with $|I|\le \epsilon n.$ Then by replacing $A$ with $a^{-1}A$ we may assume that $U = H.$  As $A\cap H$ is $2\sqrt{n}$ global inside $H$, we may apply Lemma \ref{lem:globalness implies covering} with $A\cap H$ in place of $A$ and $H$ in place of $A_n$ to complete the proof that a product of $M$ translates of $A$ contains $H$ for an absolute constant $M$. 
\end{proof}

Let us write $U_I$ for the pointwise stabilizer of the set $I$ inside $A_n$. 

\begin{lem}\label{lem:Growth for umvirates}
    Let $I,J,K$ be disjoint subsets of $[n]$ of size $\ge 2$. Then $U_I U_J U_K = A_n.$ 
\end{lem}
\begin{proof}
    Let $\sigma \in A_n.$ 
    Then it is easy to see that that there exists $\sigma_I\in U_I$ that agrees with $\sigma$ outside of the set $I\cup\sigma^{-1}(I)$ that also fixes every fixed point of $\sigma.$ 
    The permutation $\sigma_I^{-1}\sigma$ now fixes all the fixed points of $\sigma$ and also fixes the complement of $I\cup \sigma^{-1}(I).$ 
    We may now repeat with $\sigma\sigma_I^{-1}$ to find $\sigma_J \in U_J$, such that the fixed points of $\sigma_J^{-1}\sigma_I^{-1}\sigma$ contain both the complement of $I\cup \sigma^{-1}(I)$ as well as the complent of $J \cup \sigma^{-1}(J)$. Now every element of $K$ is not in $I\cup J$ and it cannot be in both $\sigma^{-1}(I)$ and $\sigma^{-1}(J)$ and therefore $\sigma_J^{-1}\sigma_I^{-1}\sigma\in U_K$ completing the proof.  
\end{proof}

\begin{thm}\label{prop:Lifshitz-AgB-Flattening}
There exists absolute constants $c,M>0$, such that the following holds. Let $G = A_n$ and suppose that $A\subseteq G$ is with $|A|>|G|^{5/6}.$ Then there exist $\sigma_1, \ldots, \sigma_M\in G$, such that \[A^{\sigma_1}A^{\sigma_2}\ldots A^{\sigma_M} = G.\]
\end{thm}
\begin{proof}
By Lemma \ref{lem:reduction to umvirates} there are $M_1$ translates of $A$ whose product contains the pointwise stabilizer of a set $I$ of size $\le n/3.$ By Lemma \ref{lem:Growth for umvirates} there are 3 conjugates of $U_I$ whose product is $H$. Combining this facts yields that there are $3M_1$ translates of $A$ whose product is $G$ and therefore that are also $3M_1$ conjugates of $A$ whose product is $G$. 
\end{proof}

\begin{proof}[Proof of \fref{thm:Filling-with-very-large-sets} for
  the Alternating groups]
  Let $M'$ be the absolute constant named $M$ in
  \fref{prop:Lifshitz-AgB-Flattening},
  and set $M=200M'$.
  We choose $\delta$ so small that
  $\delta\le\frac{1}{200M}$.
  Let $G$ denote the alternating group
  and let $A_1\dots,A_M$ be subsets of $G$
  of size at least $|G|^{1-\delta}$.
    
  If $X,Y$ are subsets of $G$
  then
  \begin{equation}
    \label{eq:5}
    |X\cap hY|\ge |X|\cdot\frac{|Y|}{| G|}
    \quad\quad
    \text{for some }h\in  G.
  \end{equation}
  Applying \fref{eq:5} iteratively, we can find elements
  $h_2,\dots,h_M$ in  $G$ such that the set
  \[
    A= A_1\cap h_2 A_2\cap\cdots\cap h_M  A_M
  \]
  has size at least $| G|^{5/6}$.
  Applying \fref{prop:Lifshitz-AgB-Flattening} to the set $A$
  we obtain $\sigma_1,\dots\sigma_{M-1}\in \tilde G$ such that
  \[
    G = A\sigma_1A\sigma_2\cdots A\sigma_{M-1}A \subseteq
     A_1(\sigma_1h_2) A_2(\sigma_2h_3)\cdots
     A_{M-1}(\sigma_{M-1}h_M)A_M.
  \]
  As there are conjugates of the sets $A_i$ whose product is $G$, there are also conjugates of the sets $A_i$ whose product is $G$.
\end{proof}

\section{Alternative proof of \texorpdfstring{\fref{thm:Filling-with-very-large-sets}}{Theorem 5} for groups of Lie Type}
\label{sec:Filling-very-large-Lie-type}

In this section we give an alternative proof of \fref{thm:Filling-with-very-large-sets} for the case where $S$ is a simple group of Lie type. The statement we need is the following.

\begin{thm}
  \label{thm:main-Lie-type}
  There are absolute constants $\delta>0$ and $\nu>0$
  with the following property.
  Let $L$ be a finite simple group of Lie type, and
  $X_1,\dots,X_\nu$ be subsets of $L$ of size at least
  $|L|^{1-\delta}$.
  Then there is a product decomposition
  $$
  L = X_1^{g_1}\cdots X_\nu^{g_\nu}
  \quad\quad
  \text{for some }g_1,\dots,g_\nu\in L.
  $$
\end{thm}

 A key ingredient to proving \fref{thm:main-Lie-type} is the following linear algebraic analogue, which is of interest in its own right. Here, and below, we write $\Mat$ to mean the additive group of $n\times n$ matrices with
  entries in the finite field $\BF_q$.

\begin{thm}\label{thm:main-additive}
  There is an absolute constant $\mu$ with the following property.
   Let $X_1,\dots,X_\mu\subseteq\Mat$ be a subsets of size at least
   $q^{\frac9{10}n^2}$.
   Then there are matrices $a_1,b_1,\dots,a_\mu,b_\mu\in\GL(n,q)$
   such that
   $$
   a_1^{-1}X_1b_1 +\dots+a_\mu^{-1} X_\mu b_\mu=\Mat
   $$
\end{thm}

We will see in Subsection~\ref{subsec:groups-lie-type} that \fref{thm:main-Lie-type} is an easy consequence of \fref{thm:main-additive}. Much more difficult is the proof of \fref{thm:main-additive} itself. An important result along the way to proving this result is stated next; it is the third main result of this section.

To state the result, we need two pieces of notation. First, we denote by $\RRR$ the number of matrices of rank $r$ in $\Mat$. Second, for a matrix $m\in\Mat$ of rank $t$ we denote by
  $\NNN(r_1,\dots,r_k\,;t)$ the number of $k$-tuples of matrices
  $(a_1,\dots,a_k)\in(\Mat)^k$ of ranks $r_1,\dots,r_k$ respectively whose
  sum is $m$. (This number is independent of $m$.) Now the result we prove is the following.

\begin{thm} \label{thm:k=3-new}
  Let $n>0$, $0\le r_1,r_2,r_3,t\le n$
  be integers.  If $r_1,r_2,r_3\ge\frac23n$ then
  $$
  \NNN(r_1,r_2,r_3\,;t) =
  \Ordo(1)\frac{\RRR(r_1)\,\RRR(r_2)\,\RRR(r_3)}{q^{n^2}}
  $$
  independent of $t$.
\end{thm}

Effectively, \fref{thm:k=3-new} asserts that the sum of three random matrices of large enough rank does not tend to concentrate on matrices of any particular rank.

In this section, we will first show that \fref{thm:main-additive} implies \fref{thm:main-Lie-type}; we will then show that \fref{thm:k=3-new} implies \fref{thm:main-additive}; and, finally, we will prove \fref{thm:k=3-new}.

\subsection{Conjectures}

\fref{thm:main-Lie-type} is generalized in
\fref{thm:Filling-G-ad-delta-sized-sets}
to subsets of size at least $|S|^\delta$.
It seems possible that the other two results just stated can be
generalized.
We propose the following two conjectures.

\begin{conj}
  For all $c>0$ there is an integer  $\mu=\mu(c)$ with the following property.
  Let $X_1,\dots,X_\mu\subseteq\Mat$ be
  subsets of size at least $q^{cn^2}$.
  Then
  there are matrices $a_1,b_1,\dots,a_\mu,b_\mu\in\GL(n,q)$
  such that
  $$
  a_1Xb_1 +\dots+a_\mu Xb_\mu=\Mat
  $$
\end{conj}

\begin{conj}
For all $\epsilon>0$ there exists a number $K(\epsilon)$ with the following property: Let $n>0$, $0\le s\le n$ be integers.  If $k\ge K(\epsilon)$, and
  $\epsilon n\le r_1,\dots,r_k\le n$ are integers, then
  $$
  \NNN(r_1,\dots,r_k\,;t) = \Ordo(1)\,
  \frac{\RRR(r_1)\dots\RRR(r_k)}{q^{n^2}}.
  $$
\end{conj}

\subsection{Theorem~\ref{thm:main-additive} implies Theorem~\ref{thm:main-Lie-type}}
\label{subsec:groups-lie-type}

Our first result is a variation of
\cite[Theorem~1]{Nikolov.Product.Decomposition.Classical.Groups}.
\begin{prop}
  \label{prop:Lie-type-covered-with-SL(2n+,q)(n,q)}
  There are absolute constants $r,M_1>0$ with the following property:
  if $L$ is a finite simple group of Lie type
  of Lie rank at least $r$,
  then there is a homomorphism $\phi:SL(2n+1,q)\to L$
  for some $n\ge1$ and some $q$,
  such that $L$ is the product of $M_1$ appropriate conjugates of
  $\im(\phi)$.

  Moreover, for every $N>0$, there exists $R>0$ such that if $r$, the rank of $L$, is larger than $R$ then the statement holds with $n>N$.
\end{prop}

Our proof shows that, provided the rank of $r$ is large enough, then this result is true with $M_1=1000$.

\begin{proof}
  If $r$ is large enough, then $L$ is of classical type. If $L$ is of $A_m$ type with $m+1$ odd, then the statement is automatic.  If $m+1$ is even, then, by the final remark of \cite[Section~1]{Nikolov.Product.Decomposition.Classical.Groups}, the group $SL(m+1,q)$ is the product of $5$ appropriate conjugates of the natural subgroup $SL(m,q)$ and so the result holds for $L=PSL(m+1,q)$ with $M_1=5$.

  Otherwise, by   \cite[Theorem~1]{Nikolov.Product.Decomposition.Classical.Groups}
  we obtain a homomorphism $\psi:SL(m,q)\to L$
  such that $L$ is the product of $200$ appropriate conjugates of the
  subgroup $\im(\psi)$, and the finite field $\BF_q$ is a subfield of
  the defining field of $L$. What is more, the construction given in \cite{Nikolov.Product.Decomposition.Classical.Groups} satisfies the condition that, for every $N>0$, there exists $R>0$ such that if $r$, the rank of $L,$ is larger than $R$ then we can choose $m>N$. In particular, $m\ge3$ if $r$ is large enough.

  If $m$ is odd, then the required property holds with $\phi=\psi$. If $m$ is even then we again use the final remark of \cite[Section~1]{Nikolov.Product.Decomposition.Classical.Groups} that the the group $SL(m,q)$ is the product of $5$ appropriate conjugates of the natural subgroup $SL(m-1,q)$. Our statement holds with $\phi$ being the restriction of  $\psi$ to $SL(m-1,q)$.
\end{proof}

Next, some notation.

\begin{defn}
  \label{defn:Un-Tn}
  Let $I_{a\times a}$ and $O_{a\times b}$ denote the unit-matrix and the
  zero-matrix of size indicated by the index.

  In the group $H=SL(2n+1,q)$ we define three important subgroups via block matrices:
  \begin{align*}
  U_n &=
  \left\{\
    \begin{pmatrix}
      I_{n\times n} & O_{n\times1} & a
      \\
      O_{1\times n} & I_{1\times1} & O_{1\times n}
      \\
      O_{n\times n} & O_{n\times1} & I_{n\times n}
    \end{pmatrix}
    \
    :
    \
    a\in\Mat
  \right\}; \\
  \CT_n &= \left\{\
    \begin{pmatrix}
      g & O_{n\times1} & O_{n\times n}
      \\
      O_{1\times n} & \frac1{\det(gh)} & O_{1\times n}
      \\
      O_{n\times n} & O_{n\times1} & h
    \end{pmatrix}
    \
    :
    \
    g,h\in GL(n,q)
  \right\}.
\end{align*}
\end{defn}

\begin{lem} \label{lem:Un-Tn}
  $U_n$ is isomorphic to the additive group of $\Mat$
  and
  $\CT_n$ is isomorphic to $GL(n,q)\times GL(n,q)$.
  Moreover, $\CT_n$ normalises $U_n$, and the conjugation action
  of $\CT_n$ on $U_n$
  is given by the formula
  $$
  (g,h): a \to g^{-1}ah
  \quad\quad
  \text{for }g,h\in GL(n,q),\,a\in\Mat.
  $$
\end{lem}
\begin{proof}
  Matrix multiplication, left to the reader.
\end{proof}

Finally, two results about the group $SL(2n+1,q)$.

\begin{lem}\label{lem: use ud}
The group $H=SL(2n+1,q)$ is a product of 30 conjugates of the group $U_n$.
\end{lem}
\begin{proof}
  Let $P$ be the Sylow $p$-subgroup consisting of all unipotent upper
  triangular matrices.
  By \cite{bnp}, we know that $H$ is a product of 5 conjugates of $P$.
  To complete the proof, we need only prove that $P$ is contained in
  the product of 6 conjugates of $U_n$.

  For  directed graphs $\gamma$ on the vertex set
  $\{1,2,\dots,2n+1\}$
  we define
  \[
    \CM(\gamma) =
    \left\{
      M\in\Mat[2n+1] \;\Bigg|\;
      { M_{i,i}=1
        \text{ and } M_{i,j}=0 \text{ for all } j\ne i 
        \atop
        \text{ unless there is an edge
        }(i,j) \text{ in } \gamma }
    \right\}
  \]
  If another graph $\gamma'$ is isomorphic to a subgraph of $\gamma$
  then the permutation matrix corresponding to this isomorphism
  conjugates $\CM(\gamma')$ into a subset of $\CM(\gamma)$.
  Consider the following graphs on $\{1,2,\dots,2n+1\}$:
  \begin{align*}
    \alpha
    &= \left\{ (2i-1,2j) \,\Big|\,1\le i\le j\le n\right\},
    \\[2pt]
    \kappa
    &= \left\{ (2i,2j+1) \,\Big|\,1\le i\le j\le n\right\},
    \\[2pt]
    \upsilon
    &= \left\{ (i,j) \,\Big|\,1\le i\le n,\ n+2\le j\le2n+1\right\}.
  \end{align*}  
  Observe, that
  $\upsilon$ is a complete bipartite graph on $n+n$ vertices,
  and by definition $\CM(\upsilon)=U_n$.
  Since $\alpha$ and $\kappa$ are also bipartite graphs on $n+n$ vertices,
  they are isomorphic to appropriate subgraphs
  of $\upsilon$.
  So $\CM(\alpha)$ and $\CM(\kappa)$ are subsets
  of two appropriate conjugates of $\CM(\upsilon)$,
  and it is enough to show that
  \[
    \big(\CM(\alpha)\cdot\CM(\kappa)\big)^3 = P,
  \]
  i.e. for each unipotent upper triangular matrix $T\in P$
  we need to find matrices
  $A,B,C\in\CM(\alpha)$ and
  $K,L,M\in\CM(\kappa)$ such that
  \begin{equation}
    \label{eq:2}
    T = AKBLCM.
  \end{equation}
  We consider $T$ fixed,
  and solve
  \fref{eq:2}
  for the matrix entries on the right hand side.
  In fact most of the entries are already fixed (and equal to 0 or 1)
  by the definition of the graphs $\alpha$ and $\kappa$.
  We reduce further the number of variables by setting
  \[
    B_{2i-1,2j} =1,
    \quad
    L_{2i,2j+1} = 1 - K_{2i,2j+1}
    \quad
    \text{for all }1\le i\le j\le n,
  \]
  and for convenience we introduce a change of variables:
  \[
    \tilde A_{2i-1,2i} = A_{2i-1,2i} + C_{2i-1,2i}
    \quad
    \text{for all }1\le i\le j\le n.
  \]
  After incorporating these changes,
  we need to solve \fref{eq:2} for the variables
  \begin{equation}
    \label{eq:3}
    \tilde A_{2i-1,2j}, C_{2i-1,2j}, K_{2i,2j+1}, M_{2i,2j+1}
    \quad
    \text{for all }1\le i\le j\le n.
  \end{equation}
  We order the indices of the matrix entries from down upward, and
  from left to right, i.e. we use the order
  \[
    (s',t')\prec(s,t)
    \quad\Longleftrightarrow\quad
    s'>s,\text{ or }\, s'=s\text{ and }t'<t.
  \]
  In the following calculations $\Omega_{s,t}$ will denote an
  arbitrary polynomial of those matrix entries of
  $A,B,C,K,L,M$ whose indices precede $(s,t)$.
  In different expressions $\Omega_{s,t}$ may refer to different
  polynomials.

  A simple formula for the product of 6 triangular matrices:
  \[
    T_{s,t} = \sum_{s\le l_1\le l_2\le l_3\le l_4\le l_5\le t}
    A_{s,l_1}K_{l_1,l_2}B_{l_2,l_3}L_{l_3,l_4}C_{l_4,l_5}M_{l_5,t}
  \]
  Since the diagonal entries in our matrices are 1,
  the formula further simplifies:
  we simply omit all diagonal elements from the 6-fold product.
  
  The matrix multiplications in
  \fref{eq:2}
  gives us a single equation for each $T_{s,t}$
  for all  $1\le s<t\le 2n+1$,
  whose leading term is particularly simple,
  and the other terms we merge into an $\Omega$-term.
  Depending on the parity of the indices,
  these equations have one of the the following four forms:
  \begin{eqnarray*}
    T_{2i-1,2j}
    &=& A_{2i-1,2j} + B_{2i-1,2j} +  C_{2i-1,2j} + \Omega_{2i-1,2j} \\
    &=& \tilde A_{2i-1,2j} + 1 + \Omega_{2i-1,2j},
    \\[8pt]
    T_{2i-1,2j+1}
    &=& A_{2i-1,2j} (K_{2j,2j+1}+L_{2j,2j+1}+M_{2j,2j+1}) +
    \\ &&  B_{2i-1,2j} (L_{2j,2j+1} + M_{2j,2j+1}) +
          C_{2i-1,2j} M_{2j,2j+1} + \Omega_{2i-1,2j}
    \\ &=& \tilde A_{2i-1,2j}(1+M_{2j,2j+1}) - C_{2i-1,2j} +
           L_{2j,2j+1} + M_{2j,2j+1} + \Omega_{2i-1,2j}
    \\ &=& \tilde A_{2i-1,2j}(1+M_{2j,2j+1}) - C_{2i-1,2j} + \Omega_{2i-1,2j},
    \\ && \kern -20pt
          \text{where the  }L_{2j,2j+1}\text{ and }M_{2j,2j+1}
          \text{ terms are merged  into }\Omega_{2i-1,2j},
    \\[8pt]
    T_{2i,2j+1}
    &=& K_{2i,2j+1} + L_{2i,2j+1} + M_{2i,2j+1} + \Omega_{2i,2j+1} \\
    &=& 1 +  M_{2i,2j+1} + \Omega_{2i,2j+1},
    \\[8pt]
    T_{2i,2j+2}
    &=& K_{2i,2j+1}(B_{2j+1,2j+2} + C_{2j+1,2j+2}) +
        L_{2i,2j+1}C_{2j+1,2j+2} + \Omega_{2i,2j+1} \\
    &=& K_{2i,2j+1} + C_{2j+1,2j+2} + \Omega_{2i,2j+1}
    \\ &=& K_{2i,2j+1} + \Omega_{2i,2j+1},
    \\ && \kern -20pt
          \text{where the  }C_{2j+1,2j+2}
          \text{ term is merged  into }\Omega_{2i,2j+1}.  
  \end{eqnarray*}
  We order the $2n+1\choose2$ equations according to the index of the
  $T$-term on the left hand side.
  Each equation contains a variable with coefficient $\pm1$
  (i.e. $\tilde A_{2i-1,2j}$, $-C_{2i-1,2j}$, $M_{2i,2j+1}$, $K_{2i,2j+1}$) 
  which do not appear in any of the preceding equations.
  This implies that the system has a solution.
  Indeed, we start out by setting arbitrary initial values, say 0,
  to all the variables in \fref{eq:3}.
  Then we go through the equations one by one, in order,
  and make them hold by adjusting the value of the corresponding
  variable. This adjustment does not effect the validity of the
  preceding equations, so at the end all equations will hold.
\end{proof}

\begin{proof}[Proof that \fref{thm:main-additive} implies \fref{thm:main-Lie-type}]
  For groups of bounded Lie rank, the statement is a special case of
  \cite[Theorem~2.]{Gill-Pyber-Szabo-2019-Roger-Saxl-bounded-rank}.
  So we may assume that the rank of $L$ is large.
  In particular,
  \fref{prop:Lie-type-covered-with-SL(2n+,q)(n,q)}
  gives us a homomorphism
  $$
  \phi:H\to L,
  \quad\quad
  \text{where \ }H = SL(2n+1,q)
  $$
  such that $L$ is the product of $M_1=1000$ appropriate conjugates of $\phi(H)$. What is more, by the final assertion of \fref{prop:Lie-type-covered-with-SL(2n+,q)(n,q)}, we can assume that $n$ is large.
  Let $U_n\le H$ be the subgroup defined in \fref{defn:Un-Tn}.
  Then
  $$
  |U_n|\ge|H|^{\frac19}\ge|L|^{\frac1{9M_1}}.
  $$
  \fref{lem: use ud} implies that $H$ is equal to the product of $M_2=75$ appropriate conjugates of $U_n$. Combining these we obtain that $L$ is the product of $m = M_1M_2$ appropriate conjugates of $\bar U_n=\phi(U_n)$.

  An easy calculation allows us to generalize this to a statement about cosets of $\bar U_n$ and we obtain:
  \begin{fact}
    \label{fact:conjugates-of-cosets-cover-SL}
    Given $m$ cosets of $\bar U_n=\phi(U_n)$ in $L$, say
    $\gamma_0\bar U_n,\dots,\gamma_{m-1}\bar U_n$,
    there are conjugates
    $(\gamma_0\bar U_n)^{s_0},\dots,(\gamma_{m-1}\bar U_n)^{s_{m-1}}$
    whose product is the whole of $L$.
  \end{fact}
  We set
  $$
  \nu=m\mu
  $$
  where $\mu$ is the absolute constant defined in \fref{thm:main-additive}.

  For each $i=1,\dots,\nu$ we choose an element $x_i\in L$ so that
  $|X_i\cap x_i\bar U_n|$ is maximal, and define
  $Y_i=\phi^{-1}\big(x_i^{-1}X_i\cap \bar U_n\big)$.
  Then
  $$
  |Y_i|\ge
  \frac{|U_n|}{|\bar U_n|}\cdot\frac{|X_i|}{|L:\bar U_n|}\ge
  \frac{|L|^{1-\delta}\cdot|U_n|}{|L|}\ge|U_n|^{1-9\delta M_1}
  $$
  We partition the sequence $Y_1\dots Y_\nu$ into $m$ subsequences of
  length $\mu$:
  $$
  Y_{j\mu+1},\dots,Y_{(j+1)\mu}
  \quad\quad
  \text{for }j=0,\dots,m-1.
  $$
  We apply \fref{thm:main-additive} to each of these subsequences.
  We can do that if $\delta$ is sufficiently small, e.g.,
  $\delta = \frac1{90M_1}=\frac{1}{90000}$.
  We obtain elements $t_1\dots t_\nu\in\CT_n$ such that
  (with multiplicative notation)
  $$
  U_n=Y_{j\mu+1}^{t_{j\mu+1}}\cdots Y_{(j+1)\mu}^{t_{(j+1)\mu}}
  \quad\quad
  \text{for }j=0,\dots,m-1.
  $$
  Combining this with the definition of $Y_i$ we get
  $$
  U_n\subseteq
  \phi^{-1}\big(x_{j\mu+1}^{-1}X_{j\mu+1}\big)^{t_{j\mu+1}} \cdots
  \phi^{-1}\big(x_{(j+1)\mu}^{-1}X_{(j+1)\mu}\big)^{t_{(j+1)\mu}}
  \quad\quad
  \text{for }j=0,\dots,m-1.
  $$
  One can rewrite it in the form
  $$
  \gamma_j\bar U_n\subseteq
  X_{j\mu+1}^{\tilde t_{j\mu+1}}\cdots X_{(j+1)\mu}^{\tilde t_{(j+1)\mu}}
  \quad\quad
  \text{for }j=0,\dots,m-1
  $$
  where
  $$
  \gamma_j =
  (x_{(j+1)\mu})^{\phi(t_{(j+1)\mu})}\cdots(x_{j\mu+1})^{\phi(t_{j\mu+1})}
  $$
  and
  $\tilde t_1,\dots,\tilde t_\nu$ are appropriate elements of $L$.
  Now the result follows from \fref{fact:conjugates-of-cosets-cover-SL}.
\end{proof}

\subsection{Theorem~\ref{thm:k=3-new} implies Theorem~\ref{thm:main-additive}}\label{sec:second-imp}

We establish some notation for this subsection: $A$ always denotes an abelian group with additive notation. Automorphisms always act on the left. We set $\CT$ to be a subgroup of $\Aut(A)$.  For elements $a\in A$, $t\in\CT$, subsets $X,Y\subseteq A$, and integer
  $m$ we define the following subsets of $A$:
  \begin{eqnarray*}
    X+Y &=& \big\{x+y\,\big|\,x\in X,\,y\in Y\big\}\\
    m X &=& X+\dots+X \quad m\text{-fold sum}\\
    t X &=& \big\{tx \,\big|\,x\in X\big\}\\
    \CT a &=& \big\{ta \,\big|\,t\in\CT\big\}\\
    \CT X &=& \big\{tx \,\big|\,t\in\CT,\,x\in X\big\}
  \end{eqnarray*}
  We say that a subset $X\subseteq A$ is \emph{$\CT$-invariant} if $\bigcup\limits_{t\in CT}t(X) = X$.

  We should also remind ourselves of some notation that was introduced at the start of this section. We denote by $\RRR$ the number of matrices rank $r$ in $\Mat$. For matrices $m\in\Mat$ of rank $t$ we denote by
  $\NNN(r_1,\dots,r_k\,;t)$ the number of $k$-tuples
  $a_1,\dots,a_k\in\Mat$ of ranks $r_1,\dots,r_k$ respectively whose
  sum is $m$. (This number is independent of $m$.) This notation will be used again in the next subsection.

\subsection{Olson's Theorem}
\label{subsec:olsons-theorem}

Olson in \cite{olson.sum-growth} proved the following generalisation of Kneser's theorem.
\begin{fact}[Olson]
  \label{fact:Olson-Small-growth}
  Let $C =  AB$, where $A$ and $B$ are finite subsets of a
  group $G$ and $1\in B$. Then either $C B = C$
  or  $|AB| \ge |A|+\frac12|B|$.
\end{fact}

\begin{cor}
  \label{cor:Olsons-non-abelian}
  Let $X_1,\dots,X_l$ be subsets of a finite group $G$ containing $1$.
  Let $\CT$ be a subgroup of $Aut(G)$.
  Let $H_i$ denote the $\CT$-closure of $X_i$,
  i.e. the subgroup generated by the $\CT$-images of $X_i$.
  Then there are elements
  $t_1,\dots,t_l$ of $\CT$, such that
  \begin{enumerate}[(a)]
  \item
    either $\big|t_1(X_1)\cdots t_l(X_l)\big| \ge
    \frac12\big(|X_1|+\cdots+|X_l|\big)$
  \item 
    or there exists an index $i$ such that
    $t_1(X_1)\cdots t_i(X_i)$ is a union of left cosets of $H_{i+1}$
    for some $i$.
  \end{enumerate}
\end{cor}
\begin{proof}
  This follows by induction from \fref{fact:Olson-Small-growth}.
\end{proof}

\begin{cor}
  \label{cor:constant-growth-or-filling}
  Let $X_1,\dots,X_l$ be subsets of a finite abelian group $A$
  and let $x_i\in X_i$ be arbitrary elements for all $i$.
  Let $\CT$ be a subgroup of $\Aut(A)$.
  Then there are elements $t_1,\dots,t_l$ of $\CT$
  such that
  \begin{enumerate} [(a)]
  \item \label{item:abelian-growth}
    either $|t_1X_1+\dots t_lX_l| \ge \frac l2\min\big(|X_1|,\dots,|X_l|\big)$,
  \item \label{item:abelian-filling}
    or $t_1X_1+\dots+t_lX_l$ is a union of cosets
    of some $\CT$-invariant subgroup $N$ of size
    $|N|\ge\min\big(|X_1|,\dots,|X_l|\big)$.
  \end{enumerate}
\end{cor}
\begin{proof}
  Apply \fref{cor:Olsons-non-abelian} to the sets $X_j-x_j$.
  If \fref{item:abelian-growth} fails then
  $t_1X_1+\dots+t_{i-1}X_{i-1}$ is a union of cosets
  of the $\CT$-closure of $X_i-x_i$ for some $i$.
  This implies \fref{item:abelian-filling}.
\end{proof}

\subsection{Additive Energy}
\label{subsec:additive-energy}

The following notion apears in
\cite[Section~6.]{Razborov2014product.thm.in.free.groups},
it is called there ``collision number''.
It is a straightforward generalization of
\cite[Definition~2.8]{Tao-Vu},
where it is called ``additive energy''.
We prefer to keep the second name.
\begin{defn}
  Let $A$ be an Abelian group.  The \emph{additive energy} of $k$
  subsets $X_1,\dots,X_k\subseteq A$ is
  $$
  E(X_1,\dots X_k) = \Big|\big\{(x_1,x_1',\dots,x_k,x_k') \in \prod
  X_i^2 \,\big|\, \sum x_i=\sum x_i' \big\}\Big|.
  $$ 
\end{defn}

The following lemma apears at the beginning of 
\cite[Section~6.]{Razborov2014product.thm.in.free.groups},
It is a generalisation of \cite[Corollary~2.10]{Tao-Vu}.
\begin{lem}\label{lem:additive-energy-growth}
  Let $A$ be an Abelian group with additive notation, and
  $X_1,\dots,X_k\subseteq A$ subsets for some $k\ge2$.  Then
  $$
  \Big|\sum X_i\Big|\ge\frac{\prod|X_i|^2}{E(X_1,\dots,X_k)}
  $$
\end{lem}
\begin{proof}
  \begin{eqnarray*}
    E(X_1,\dots,X_k)
    &=&
        \sum_{z\in\sum X_i}\Big|\big\{(x_1,\dots,x_k)\in\prod X_i
        \,\big|\,\sum x_i = z\big\}\Big|^2
    \\&\ge&
            \frac{
            \left(\sum_{z\in\sum X_i}
            \Big|\big\{(x_1,\dots,x_k)\in\prod X_i
            \,\big|\,\sum x_i = z\big\}\Big|\right)^2
            }
            {\big|\sum X_i\big|}
    \\&=&
          \frac{\big|\prod X_i\big|^2}{\big|\sum X_i\big|} =
          \frac{\prod |X_i|^2}{\big|\sum X_i\big|}
  \end{eqnarray*}
\end{proof}

\begin{lem}\label{lem:growth-with-automorphisms}
  Let $A$ be an Abelian group, and let $\CT$ be a finite subgroup of
  $\Aut(A)$.  Suppose that for some $\alpha,\beta>0$ and some
  $l\in\BN$ there is a subset $V\subseteq A$ of size $|V|\le\beta$
  such that
  $$
  \Big|\big\{(t_1,\dots,t_l)\in\CT^l\,\big|\, a=t_1b_1+\dots
  t_lb_l\big\}\Big|\le\frac{|\CT|^l}\alpha
  $$
  for all $a\in A$ and all $b_1,\dots,b_l\in A\setminus V$.  Then for
  any $k\ge l$ and any $k$-tuple of subsets $X_1,\dots,X_k\subseteq A$
  with size $|X_1|=\dots=|X_k|=\chi$ there are elements
  $t_1,\dots t_k\in\CT$ such that
  $$
  \left|\sum t_iX_i\right| \ge \frac\alpha {{k\choose l} + {k\choose
      l-1}\:\frac{\alpha\beta^{k-l+1}}{\chi^{k-l+1}}}
  $$%
\end{lem}

Our argument is motivated by \cite[Exercises 2.8.4 and 2.8.5]{Tao-Vu}.
\begin{proof}
  Let $Z=X_1\times\dots\times X_k$.  We use vector-notation
  $\underline x$ and $\underline t$ for the $k$-tuples
  $(x_1,\dots,x_k)\in Z$ and $(t_1,\dots,t_k)\in\CT^k$.  Consider the
  set
  $$
  \CH = \big\{(\underline x, \underline x',\underline t)\in Z\times
  Z\times\CT^k \, \big|\ \sum t_i(x_i-x'_i)=0 \big\}.
  $$%
  It is easy to see that
  $$
  \big|\CH\big| = \sum_{\underline t\in\CT^k}E(t_1X_1,\dots,t_kX_k)
  $$
  For each subset $J\subseteq\{1,\dots,k\}$ of size $|J|=l$ we study
  the subset
  $$
  \CH_J = \Big\{(\underline x, \underline x',\underline t)\in\CH
  \,\Big|\, x_j-x'_j\notin V\text{ for }j\in J\Big\}
  $$
  Let $j_1,\dots,j_l$ denote the elements of $J$, and let
  $\{i_1,\dots,i_{k-l}\}$ be the complement of $J$.  For each tuple
  $$
  (\underline x,\underline x',t_{i_1}\dots,t_{i_{k-l}}) \in Z\times
  Z\times\CT^{k-l}
  $$
  such that $x_j-x'_j\notin V$ for all $j\in J$ there are at most
  $\frac{|\CT|^l}\alpha$ possible choices of
  $(t_{j_1},\dots,t_{j_l})\in\CT^l$ such that
  $$
  \sum_{j\in J}t_j(x_j-x'_j)=-\sum_{i\notin J}t_i(x_i-x'_i).
  $$%
  Therefore
  $$
  \big|\CH_J\big|\le \big|Z\times
  Z\times\CT^{k-l}\big|\;\frac{|\CT|^l}\alpha =
  \frac{\chi^{2k}}\alpha\,|\CT|^k.
  $$%
  Next, for each $I\subseteq\{1,\dots,k\}$ of size $|I|=k-l+1$ we
  consider the subset
  $$
  \tilde\CH_I = \Big\{(\underline x,\underline x',\underline t)\in \CH
  \,\Big|\, x_i-x'_i\in V \ \text{ for all } i\in I \Big\}.
  $$
  For each $i\in I$ and for each $x_i\in X_i$ there are at most $|V|$
  possible choices of $x'_i$ such that $x_i-x'_i\in V$.  Hence
  $$
  \big|\tilde\CH_I\big| \le \left(\prod_{i\in
      I}\big(|X_i|\cdot|V|\big)\right)\left(\prod_{j\notin
      I}|X_j|^2\right)|\CT|^k \le \beta^{k-l+1}\chi^{k+l-1}\,|\CT|^k.
  $$%
  This implies that
  \begin{eqnarray*}
    \sum_{\underline t\in\CT^k}E(t_1X_1,\dots,t_kX_k)
    &=&
        \left|\Big(\bigcup_J\CH_J\Big)\bigcup\Big(\bigcup_I\tilde\CH_I\Big)\right|
    \\&\le&
            |\CT|^k\:\left({k\choose l}\:\frac{\chi^{2k}}\alpha +
            {k\choose k-l+1}\:\beta^{k-l+1}\chi^{k+l-1}\right)    
  \end{eqnarray*}
  Hence for some $\underline t\in\CT^k$ we have
  $$
  E(t_1X_1,\dots,t_kX_k)\le {k\choose l}\:\frac{\chi^{2k}}\alpha +
  {k\choose k-l+1}\:\beta^{k-l+1}\chi^{k+l-1}
  $$
  Now \fref{lem:additive-energy-growth} implies that
  $$
  \left|\sum t_iX_i\right| \ge \frac{\chi^{2k}} {{k\choose
      l}\:\frac{\chi^{2k}}\alpha + {k\choose
      k-l+1}\:\beta^{k-l+1}\chi^{k+l-1}} = \frac\alpha {{k\choose l} +
    {k\choose l-1}\:\frac{\alpha\beta^{k-l+1}}{\chi^{k-l+1}}}
  $$
\end{proof}

\begin{cor}\label{cor:growth-with-endomorphisms}
  For all $\epsilon>0$ and all $l\in\BN$
  there is a positive integer $\nu=\nu(\epsilon,l)$
  with the following property.
  Let $A$ be a finite Abelian group, and let $\CT$ be a subgroup of
  $\Aut(A)$.  Suppose that for some $\alpha,\beta\ge1$
  there is a subset $V\subseteq A$ of size $|V|\le\beta$
  such that
  for all $a\in A$ and all $b_1,\dots,b_l\in A\setminus V$
  we have
  $$
  \Big|\big\{(t_1,\dots,t_l)\in\CT^l\,\big|\, a=t_1b_1+\dots
  t_lb_l\big\}\Big|\le\frac{|\CT|^l}\alpha .
  $$
  Then for any sequence $X_1,\dots,X_\nu$ of subsets of $A$ satisfying
  $$
  |X_i|\ge \alpha^\epsilon\beta
  \quad\quad
  \text{for all }i
  $$
  there are $\nu$ elements $t_1,\dots t_\nu\in\CT$ such that
  \begin{itemize}
  \item
    either
    $$
    \left|\sum t_iX_i\right| \ge\alpha
    $$
  \item
    or there is a coset $a+N$ of some $\CT$-invariant subgroup
    $N\le A$
    such that
    $$
    a+N\subseteq \sum t_iX_i
    \quad\text{and}\quad
    |N|\ge\min\big(|X_1|,\dots,|X_\nu|\big)
    $$
  \end{itemize}
\end{cor}
\begin{proof}
  We set
  $$
  k=l+\left\lfloor\frac1\epsilon\right\rfloor,
  \quad\quad
  \nu=k\cdot2^{k+1}.
  $$
  Let $X_1,\dots,X_\nu$ be as in the statement.
  By possibly shrinking the sets $X_i$, and slightly increasing
  $\beta$,
  we reduce our statement to the case when $\chi=\alpha^\epsilon\beta$
  is an integer, and
  all $X_i$ have size $\chi$.

  We partition $X_1,\dots X_\nu$
  into $2^{k+1}$ subsequences of length $k$:
  $$
  X_{jk+1},\dots,X_{(j+1)k}
  \quad\quad
  \text{for }j=0,\dots,2^{k+1}-1.
  $$
  We apply \fref{lem:growth-with-automorphisms} to
  each of these subsequences
  and we obtain
  $s_1,\dots,s_\nu\in\CT$ such that
  for all $j$ we have  
  $$
  \left|\sum_{i=jk+1}^{(j+1)k} s_iX_i\right|
  \ge
  \frac\alpha{{k\choose l} +
    {k\choose l-1}\:\frac{\alpha\beta^{k-l+1}}{(\alpha^\epsilon\beta)^{k-l+1}}}
  \ge
  \frac\alpha{{k\choose l} + {k\choose l-1}}
  \ge
  \frac\alpha{2^{k}} .
  $$
  Applying \fref{cor:constant-growth-or-filling} to the sets
  $$
  Y_j =
  \sum_{i=jk+1}^{(j+1)k} s_iX_i
  \quad\quad
  \text{for }j=0,\dots,2^{k+1}-1
  $$
  we obtain our statement.
\end{proof}

\subsection{Counting matrices}
\label{subsec:counting-matrices}

\begin{defn}
  We treat the elements $\BF_q^n$ as column vectors, so $\Mat$ acts on
  $\BF_q^n$ via matrix multiplication from the left hand side.  With a
  slight abuse of notation we identify the matrices in $\Mat$ with the
  corresponding linear transformations $\BF_q^n\to\BF_q^n$.
\end{defn}

\begin{defn}
  We denote by $\RRR$ the number of matrices rank $r$ in $\Mat$.
\end{defn}

\begin{defn}
  For matrices $m\in\Mat$ of rank $t$ we denote by
  $\NNN(r_1,\dots,r_k\,;t)$ the number of $k$-tuples
  $a_1,\dots,a_k\in\Mat$ of ranks $r_1,\dots,r_k$ respectively whose
  sum is $m$. (This number is independent of $m$.)
\end{defn}

The goal of this section is to prove the following.

\begin{lem} \label{lem:k=3-new}
  Let $n>0$, $0\le r_1,r_2,r_3,t\le n$
  be integers.  If $r_1,r_2,r_3\ge\frac23n$ then
  $$
  \NNN(r_1,r_2,r_3\,;t) =
  \Ordo(1)\frac{\RRR(r_1)\,\RRR(r_2)\,\RRR(r_3)}{q^{n^2}}
  $$
  independent of $t$.
\end{lem}

Actually, it is interesting to ask the following:

\begin{question}
  Is there a bound $K(\epsilon)$ for all $\epsilon>0$
  with the following property?\\
  Let $n>0$, $0\le s\le n$ be integers.  If $k\ge K(\epsilon)$, and
  $\epsilon n\le r_1,\dots,r_k\le n$ are integers, then
  $$
  \NNN(r_1,\dots,r_k\,;t) = \Ordo(1)\,
  \frac{\RRR(r_1)\dots\RRR(r_k)}{q^{n^2}}
  $$
\end{question}

We begin with some well-known bounds.
\begin{fact} \label{fact:number-of-matrices}
  \begin{enumerate}[(a)]
  \item The number of injective homomorphisms
    $\BF_q^d\hookrightarrow\BF_q^D$ is
    $$
    \scrI_q(d,D) = (q^D-1)(q^D-q)\cdots(q^D-q^{d-1}),
    $$
    $$
    \frac14\,q^{dD} \le \scrI_q(d,D) \le q^{dD}.
    $$
  \item The number of $d$-dimensional subspaces in $\BF_q^D$ is
    $$
    \scrS_q(d,D) = \frac{\scrI_q(d,D)}{\scrI_q(d,d)} =
    \frac{(q^D-1)(q^D-q)\cdots(q^D-q^{d-1})}
    {(q^d-1)(q^d-q)\cdots(q^d-q^{d-1})},
    $$
    $$
    q^{d(D-d)} \le \scrS_q(d,D) \le 4\,q^{d(D-d)}.
    $$
  \item The number of matrices in $\Mat$ of rank $r$ is
    $$
    \RRR(r) = \scrS_q(n-r,n)\,\scrI_q(r,n),
    $$
    $$
    \frac1{4}\,q^{r(2n-r)} \le \RRR(r) \le 4\,q^{r(2n-r)}.
    $$
  \item The number of matrices in $\Mat$ of rank at most $r$ is at
    most
    $$
    \Big|\Hom(\BF_q^n,\BF_q^r)\Big|\scrS_q(r,n) \le 4q^{r(2n-r)}
    $$
  \end{enumerate}
\end{fact}
\begin{proof}
  The equalities are clear.  For the inequalities it is enough to
  prove that for all $n\ge1$
  $$
  a_n = \Big(1-\frac1q\Big)\Big(1-\frac1{q^2}\big) \cdots
  \Big(1-\frac1{q^n}\Big) \ge \frac14 + \frac1{2^{n+1}}
  $$
  We prove this by induction on $n$.  It holds for $n=1$.  For $n>1$,
  by the induction hypothesis
  $$
  a_{n+1} = a_n\Big(1-\frac1{q^{n+1}}\Big) \ge
  a_n\Big(1-\frac1{2^{n+1}}\Big) > \Big(\frac14 +
  \frac1{2^{n+1}}\Big)\Big(1-\frac1{2^{n+1}}\Big) =
  $$
  $$
  = \frac14+\frac1{2^{n+2}}\Big(2-\frac12-\frac1{2^n}\Big) >
  \frac14+\frac1{2^{n+2}}.
  $$
  This completes the induction step.
\end{proof}

We recall the following useful observation:
\begin{fact}\label{fact:rank-of-sum}
  For matrices $a,b\in\Mat$ we have
  $$
  \Big|\rank(a)-\rank(b)\Big|\le\rank(a+b)\le\rank(a)+\rank(b)
  $$
  Hence if $t<|r-s|$ or $t>r+s$ then
  $$
  \NNN(r,s;t)=0
  $$
\end{fact}
\begin{proof}
  Clear.
\end{proof}

We will use several times the following easy estimation.
\begin{lem} \label{lem:quadric-geometric-series}
  Let $a\le b$ and
  $q\ge2$ be integers, $c<0$ and $M$ real numbers.  Let $F(x)$ be a
  quadric polynomial of $x$ with leading coefficient $c$.  Suppose
  that $F(x)\le M$ on the interval $[a,b]$.  Then
  $$
  \sum_{x=a}^bq^{F(x)} \le \frac2{1-2^{c}}\,q^{M}
  $$
\end{lem}
\begin{proof}
  Let $x_*$ and $x_{\max}$ denote the locations of the global maximum
  of $F(x)$ and the maximum of $F(x)$ on the interval $I=[a,b]$.  We
  can write
  $$
  F(x) = c(x-x_*)^2+F(x_*)
  $$
  If $x,y\in I$ are not separated by $x_{\max}$ then they are not
  separated by $x_*$ either, hence
  $$
  |F(y)-F(x)| = -c\big|(y-x_*)^2-(x-x_*)^2\big| \ge -c(y-x)^2
  $$
  Therefore
  \begin{eqnarray*}
    \sum_{x=a}^bq^{F(x)}
    &\le&
          \sum_{i=0}^{\lfloor x_{\max}\rfloor-a}q^{F(x_{\max}-i)}
          +\sum_{j=0}^{b-\lceil x_{\max}\rceil}q^{F(x_{\max}+j)}
    \\
    &\le&
          \sum_{i=0}^\infty q^{F(x_{\max})+ci^2}
          +\sum_{j=0}^\infty q^{F(x_{\max})+cj^2}
    \\
    &\le&
          2q^{F(x_{\max})}\, \sum_{i=0}^\infty \left(2^{c}\right)^i
    \\
    &\le&
          2q^M\, \frac1{1-2^{c}}
  \end{eqnarray*}
\end{proof}

Now we are ready for the first step towards \fref{lem:k=3-new}.

\begin{lem} \label{lem:r,s,t-bound}
  Let $n>0$, $0\le r, s\le n$ be
  integers.  Then
  $$
  \NNN(r,s\,;t) = \Ordo(1)\, \frac{\RRR(r)\RRR(s)}{q^{n^2}}\,
  q^{\frac14\left(2n-r-s-t\right)^2}
  $$
  Moreover, if \fbox{$r+s+t\ge2n$} then
  $$
  \NNN(r,s\,;t) = \Ordo(1)\, \frac{\RRR(r)\RRR(s)}{q^{n^2}}
  $$
\end{lem}
\begin{proof}
  For each $\rho>0$ we define the set
  $$
  \CH_\rho = \Big\{(a,b)\in\Mat\times\Mat \,\Big|\, {\rank(a)=r,\,
    \rank(b)=s,\, a+b=m, \atop \dim\big(\ker(a)\cap\ker(b)\big)=\rho}
  \Big\}.
  $$
  We estimate $|\CH_\rho|$ in several steps.
  \begin{enumerate}[(a)]
  \item\label{item:4} The condition $a+b=m$ implies that
    $$
    \ker(a)\cap\ker(m) = \ker(b)\cap\ker(m)=\ker(a)\cap\ker(b).
    $$
  \item The subspaces $\ker(a)$, $\ker(b)$, $\ker(m)$ have dimensions
    $n-r$, $n-s$, $n-t$ respectively.
  \item\label{item:5} The condition $a+b=m$ implies that
    $$
    |s-r|\le t\le r+s,
    $$    
    The intersection of the kernels of two linear maps, say
    $\phi:X\to Y$ and $\psi:X\to Z$, is equal to the the kernel of
    $(\phi,\psi):X\to Y\oplus Z$.  So \fref{item:4} implies that
    $$
    \rho_{\min}\le \rho \le \rho_{\max}
    $$
    where
    $$
    \rho_{\min}=n -\min\big(n,r+s,r+t,s+t\big), \quad \rho_{\max}=
    n-\max(r,s,t)
    $$
  \item\label{item:6} The number of possibilities for the subspace
    $V=\ker(a)\cap\ker(m)$ in $\ker(m)$ is
    $$
    \scrS_q(\rho,n-t) \le 4\, q^{\rho(n-t-\rho)}.
    $$
  \item\label{item:7} $W_a=\ker(a)/V$ is a subspace of dimension
    $n-r-\rho$ in $\BF_q^n/V$.  For a given $V$, the number of
    possibilities for $W_a$ is at most
    $$
    \scrS_q(n-r-\rho,n-\rho) \le 4\,q^{(n-r-\rho)r}.
    $$
    (We need an upper bound only, so may safely ignore the condition
    that the intersection of \ $W_a$ and \ $\ker(m)/V$ must be zero.)
  \item\label{item:8} By symmetry, for a given $V$ the number of
    possibilities for $W_b=\ker(b)/V$ is at most
    $$
    \scrS_q(n-s-\rho,n-\rho) \le 4\,q^{(n-s-\rho)s}.
    $$
  \item\label{item:9} The subspaces $V$, $W_a$, and $W_b$ uniquely
    determine the subspaces $\ker(a)$ and $\ker(b)$.  Therefore the
    number of possible candidates for the pair of subspaces $\ker(a)$,
    $\ker(b)$ is at most
    $$
    \Big(4\, q^{\rho(n-t-\rho)}\Big) \Big(4\, q^{(n-r-\rho)r}\Big)
    \Big(4\, q^{(n-s-\rho)s}\Big) = 64\, q^{
      -\rho^2+\rho(n-r-s-t)+(nr+ns-r^2-s^2) }
    $$
  \item\label{item:10} Suppose now that the subspaces $\ker(a)$ and
    $\ker(b)$ are given.  We choose an arbitrary complement
    $T\le\BF_q^n$ to the subspace $\ker(a)+\ker(b)$.  Then
    $\dim(T)=n-(n-r)-(n-s)+\rho = r+s+\rho-n$, and the linear
    transformations $a$, $b$ are uniquely determined by their
    restrictions to the subspaces $\ker(a)$, $\ker(b)$ and $T$.
    Moreover, these restrictions satisfy the following:
    \begin{enumerate}[$\bullet$]
    \item On the subspace $\ker(a)$ we have
      $$
      a=0,\quad b=m.
      $$      
    \item On the subspace $\ker(b)$ we have
      $$
      b=0, \quad a=m.
      $$
    \item On the subspace $T$, for any $a$ there is at most one choice
      of $b$:
      $$
      b=m-a.
      $$
      (Actually $a$ can be almost arbitrary on $T$, the only condition
      being the injectivity of both $a$ and $m-a$.  Since we need an
      upper bound we may ignore this condition.)
    \end{enumerate}
    This implies, that for given subspaces $\ker(a)$ and $\ker(b)$ the
    number of possible candidates for the pair $(a,b)$ is at most
    $$
    \Big|\Hom(T,\BF_q^n)\Big| = q^{n(r+s+\rho-n)}.
    $$
  \item\label{item:11} Putting all together we obtain the bound
    \begin{eqnarray*}
      \big|\CH_\rho\big|
      &\le&
            64\, q^{
            -\rho^2+\rho(n-r-s-t)+(nr+ns-r^2-s^2)
            }
            \cdot
            q^{
            n(r+s+\rho-n)
            }
      \\
      &=&
          64\, q^{
          -\rho^2+\rho(2n-r-s-t)+(2nr+2ns-n^2-r^2-s^2)
          }
    \end{eqnarray*}
  \end{enumerate}
  
  From steps \fref{item:5} and \fref{item:11} we obtain that
  \begin{eqnarray*}
    \NNN(r,s\,;t)
    &=&
        \sum_{\rho=\rho_{\min}}^{\rho_{\max}}\big|\CH_\rho\big|
    \\
    &\le&
          \sum_{\rho=\rho_{\min}}^{\rho_{\max}}
          64\, q^{
          2nr+2ns-n^2-r^2-s^2
          }\,
          q^{
          \rho(2n-r-s-t-\rho)
          }
    \\
    &\le&
          64\, q^{
          2nr+2ns-n^2-r^2-s^2
          }
          \sum_{\rho=\rho_{\min}}^{\rho_{\max}}
          q^{
          \rho(2n-r-s-t-\rho)
          }
  \end{eqnarray*}
  Applying \fref{lem:quadric-geometric-series} and
  \fref{fact:number-of-matrices} we obtain
  \begin{eqnarray*}
    \NNN(r,s\,;t)
    &\le&
          64\, q^{
          2nr+2ns-n^2-r^2-s^2
          }\cdot
          \frac2{1-2^{-1}}\,
          q^{\left(\frac{2n-r-s-t}2\right)^2}
    \\&<&
          64\cdot4^3\,
          \frac{\RRR(r)\RRR(s)}{q^{n^2}}\,
          q^{\left(n-\frac{r+s+t}2\right)^2}
  \end{eqnarray*}
  On the other hand, if \fbox{$r+s+t\ge2n$} then
  $$
  \max_{\rho\ge0}\big(\rho(2n-r-s-t-\rho)\big)=0
  $$
  So in this case \fref{lem:quadric-geometric-series} and
  \fref{fact:number-of-matrices} imply that
  \begin{eqnarray*}
    \NNN(r,s\,;t)
    &\le&
          64\, q^{
          2nr+2ns-n^2-r^2-s^2
          }\cdot
          \frac2{1-2^{-1}}\,
          q^0
    \\&<&
          64\cdot4^3\,
          \frac{\RRR(r)\RRR(s)}{q^{n^2}}\,
  \end{eqnarray*}
  This completes the proof.
\end{proof}

Finally we go for it:

\begin{lem} [$=$ \fref{lem:k=3-new}]
  Let $n>0$, $0\le r_1,r_2,r_3,t\le n$
  be integers.  If $r_1,r_2,r_3\ge\frac23n$ then
  $$
  \NNN(r_1,r_2,r_3\,;t) =
  \Ordo(1)\frac{\RRR(r_1)\,\RRR(r_2)\,\RRR(r_3)}{q^{n^2}}
  $$
  independent of $t$.
\end{lem}

\begin{proof}
  For matrices $m\in\Mat$ of rank $t$ let $\NNNs(r_1,\dots,r_k\,;s,t)$
  denote the number of those $k$-tuples $a_1,\dots,a_k\in\Mat$ of
  ranks $r_1,\dots,r_k$ respectively such that $\rank(a_1+a_2)=s$ and
  $\sum_{i=1}^ka_i=m$.  (This number is independent of $m$.)  With
  this notation we can write
  \begin{eqnarray*}
    \NNN(r_1,r_2,r_3\,;t)
    &=&
        \sum_{s=0}^n\NNNs(r_1,r_2,r_3\,;s,t)
    \\
    &=&
        \sum_{s=0}^n\NNN(r_1,r_2\,;s)\cdot\NNN(r_3,s\,;t)
  \end{eqnarray*}
  We split this sum into four parts using the values
  $$
  s_1=2n-r_1-r_2, \quad s_2=\min(r_3+t, 2n-r_3-t), \quad
  s_3=\max(2n-r_1-r_2,2n-r_3-t)
  $$
  and estimate the partial sums using \fref{lem:r,s,t-bound} and
  \fref{lem:quadric-geometric-series}.
  
  \begin{enumerate}[\bf(a)]
    
  \item For small values of $s$, i.e. for
    $$
    s\le s_1
    $$
    \fref{lem:r,s,t-bound} gives us
    \begin{eqnarray*}
      S_1
      &=&
          \sum_{s=0}^{s_1}
          \NNN(r_1,r_2\,;s)\cdot\NNN(r_3,s\,;t)
      \\
      &=&
          \sum_{s=0}^{s_1}
          \Ordo(1)\,
          \frac{\RRR(r_1)\,\RRR(r_2)}{q^{n^2}}\,q^{\frac14(2n-r_1-r_2-s)^2}\cdot
      \\
      &&
         \cdot\ \Ordo(1)\,
         \frac{\RRR(r_3)\,\RRR(s)}{q^{n^2}}\,q^{\frac14(2n-r_3-s-t)^2}
      \\
      &=&
          \Ordo(1)\,\frac{\RRR(r_1)\,\RRR(r_2)\,\RRR(r_3)}{q^{n^2}}\,
          \sum_{s=0}^{s_1}
          q^{F(s)}
    \end{eqnarray*}
    where $F(s)$ is the following quadric polynomial of $s$:
    \begin{eqnarray*}
      F(s)
      &=&
          \frac{1}{4} \, {\left(2 \, n - r_3 - s - t\right)}^{2} +
          \frac{1}{4} \, {\left(2 \, n - r_1-r_2 - s\right)}^{2} -
          n^{2} + {\left(2 \, n - s\right)} s
      \\
      &=&
          - \frac{1}{2} \, s^{2}
          + \frac{r_1+r_2+r_3+t}{2} \, s 
          + \text{constant}
    \end{eqnarray*}
    Here ``constant'' is some polynomial of $r_1,r_2,r_3,t,n$.  Our
    plan is to estimate $S_1$ using
    \fref{lem:quadric-geometric-series}.  The leading coefficient of
    $F(s)$ is $-\frac12<0$, and the location of its global maximum is
    $s_{\max} = \frac{r_1+r_2+r_3+t}{2}$.  Using $r_i\ge\frac23n$ and
    $t\ge0$ we obtain that
    $$
    s_{\max} = \frac{r_1+r_2+r_3+t}{2} \ge n \ge 2n-r_1-r_2 = s_1
    $$
    Therefore $F(s)$ is increasing on the half-line $s\le s_1$, and
    its maximum value is
    \begin{eqnarray*}
      \max_{s\le s_1}F(s)
      &=&
          F(s_1)
      \\
      &=&
          -n^{2} + \frac{1}{4} \, \left(r_1+r_2-r_3-t\right)^{2}
          + {\left(2\,n-r_1-r_2\right)} {\left(r_1+r_2\right)}
      \\
      &=&
          -\frac14 \, {\left(2 \, n - r_{1} - r_{2} - r_{3} - t\right)}
          {\left(2 \, n - 3 \, r_{1} - 3 \, r_{2} + r_{3} + t\right)}
    \end{eqnarray*}
    Using $n\ge r_1,r_2,r_3\ge\frac23n$ and $0\le t\le n$ we obtain
    that
    $$
    2 n - r_{1} - r_{2} - r_{3} - t \le 2n - 3\cdot\frac23n-0 = 0,
    $$
    $$
    2 \, n - 3 \, r_{1} - 3 \, r_{2} + r_{3} + t \le
    2n-6\cdot\frac23n+2\cdot n =0.
    $$
    This implies that
    $$
    \max_{s\le s_1}F(s) \le 0
    $$
    From \fref{lem:quadric-geometric-series} we obtain that
    \begin{eqnarray*}
      S_1
      &=&
          \Ordo(1)\,\frac{\RRR(r_1)\,\RRR(r_2)\,\RRR(r_3)}{q^{n^2}}\,
          q^{\max_{s\le s_1}F(s)}
      \\
      &=&
          \Ordo(1)\frac{\RRR(r_1)\,\RRR(r_2)\,\RRR(r_3)}{q^{n^2}}
    \end{eqnarray*}

  \item For medium size values of $s$, i.e. for
    $$
    s_1\le s \le s_2
    $$
    \fref{lem:r,s,t-bound} gives us
    \begin{eqnarray*}
      S_2
      &=&
          \sum_{s=s_1}^{s_2}
          \NNN(r_1,r_2\,;s)\cdot\NNN(r_3,s\,;t)
      \\
      &=&
          \sum_{s=s_1}^{s_2}
          \Ordo(1)\,
          \frac{\RRR(r_1)\,\RRR(r_2)}{q^{n^2}}\cdot
          \Ordo(1)\,
          \frac{\RRR(r_3)\,\RRR(s)}{q^{n^2}}\,q^{\frac14(2n-r_3-s-t)^2}
      \\
      &=&
          \Ordo(1)\,\frac{\RRR(r_1)\,\RRR(r_2)\,\RRR(r_3)}{q^{n^2}}\,
          \sum_{s=s_1}^{s_2}
          q^{G(s)}
    \end{eqnarray*}
    where $G(s)$ is the following quadric polínomial of $s$:
    \begin{eqnarray*}
      G(s)
      &=&
          \frac{1}{4} \, {\left(2 \, n - r_3 - s - t\right)}^{2} -
          n^{2} + {\left(2 \, n - s\right)} s
      \\
      &=&
          - \frac{3}{4} \, s^{2}
          + \frac{1}{2} \, {\left(2 \, n + r_3 + t\right)} s
          + \text{constant}
    \end{eqnarray*}
    Here ``constant'' is some polynomial of $r_3,t,n$.  Our plan is
    to estimate $S_2$ using \fref{lem:quadric-geometric-series}.  The
    leading coefficient of $G(s)$ is $-\frac34<0$.  With the help of
    the function
    $$
    H(x) = x(2n-x)
    $$
    we can rewrite $G(s)$ as
    \begin{eqnarray*}
      G(s)
      &=&
          H(s)-H\left(\frac{r_3+s+t}2\right)
    \end{eqnarray*}
    Our assumption $s\le s_2$ implies that
    $$
    s\le\frac{r_3+s+t}2\le n
    $$
    Since $H(x)$ is an increasing function for $0\le x\le n$, we
    obtain that
    $$
    G(s)\le 0
    $$
    Therefore \fref{lem:quadric-geometric-series} implies that
    \begin{eqnarray*}
      S_2
      &=&
          \Ordo(1)\,\frac{\RRR(r_1)\,\RRR(r_2)\,\RRR(r_3)}{q^{n^2}}\,
          q^{\max_{s_1\le s\le s_2}G(s)}
      \\
      &=&
          \Ordo(1)\frac{\RRR(r_1)\,\RRR(r_2)\,\RRR(r_3)}{q^{n^2}}
    \end{eqnarray*}

  \item For large values of $s$, i.e. for
    $$
    s\ge s_3
    $$
    \fref{lem:r,s,t-bound} implies that
    \begin{eqnarray*}
      S_3
      &=&
          \sum_{s=s_3}^{n}
          \NNN(r_1,r_2\,;s)\cdot\NNN(r_3,s\,;t)
      \\
      &=&
          \sum_{s=s_3}^{n}
          \Ordo(1)\,
          \frac{\RRR(r_1)\,\RRR(r_2)}{q^{n^2}}\cdot
          \Ordo(1)\,
          \frac{\RRR(r_3)\,\RRR(s)}{q^{n^2}}       
      \\
      &=&
          \Ordo(1)\,\frac{\RRR(r_1)\,\RRR(r_2)\,\RRR(r_3)}{q^{n^2}}\,
          \sum_{s=s_3}^{n}
          q^{s(2n-s)-n^2}
      \\
      &=&
          \Ordo(1)\,\frac{\RRR(r_1)\,\RRR(r_2)\,\RRR(r_3)}{q^{n^2}}\,
          \sum_{s=s_3}^{n}
          q^{-(n-s)^2}
    \end{eqnarray*}
    Since $q^{-(n-s)^2}\le 1$, \fref{lem:quadric-geometric-series}
    implies that
    $$
    S_3 = \Ordo(1)\,\frac{\RRR(r_1)\,\RRR(r_2)\,\RRR(r_3)}{q^{n^2}}
    $$
    
  \item All remaining values of $s$ satisfy $2n-r_1-r_2<s<2n-r_3-t$
    and $s>s_2$, hence
    $$
    s> r_3+t
    $$
    Therefore \fref{fact:rank-of-sum} implies that
    $\NNN(r_3,s\,;t)=0$,
    and we obtain
    $$
    S_4 = \sum_{s=r_3+t+1}^n\NNN(r_1,r_2\,;s)\cdot\NNN(r_3,s\,;t) =0
    $$
    
  \end{enumerate}
  Putting the parts together we obtain that
  \begin{eqnarray*}
    \NNN(r_1,r_2,r_3\,;t)
    &\le&
          S_1+S_2+S_3+S_4
    \\
    &=&
        \Ordo(1)\,\frac{\RRR(r_1)\,\RRR(r_2)\,\RRR(r_3)}{q^{n^2}}
  \end{eqnarray*}
  
\end{proof}

\subsection{Growth among square matrices}
\label{sec:growth-square-matr}

Finally, in this section we prove our main result.
\begin{thm}[equivalent to \fref{thm:main-additive}]
  \label{thm:main-additive-copy}
  There is an absolute constant $\mu$ with the following property.
  Let $X_1,\dots,X_\mu\subseteq\Mat$ be subsets of size at least
  $q^{\frac9{10}n^2}$.
  Then there are matrices $a_1,b_1,\dots,a_\mu,b_\mu\in\GL(n,q)$
  such that
  $$
  a_1X_1b_1^{-1} +\dots+a_\mu X_\mu b_\mu^{-1}=\Mat
  $$
\end{thm}
\begin{rem}
  Our formula here is different form the one in
  \fref{thm:main-additive},
  but the two statements are clearly equivalent.
  There is a reason for this change:
  Here, in the proof we have linear transformations, and prefer to
  wright them as left actions.
  On the other hand, in he application in \fref{subsec:groups-lie-type}
  we have group conjugations, and which look more natural as right actions.
\end{rem}
\begin{proof}
  The group $\CT=\GL(n,q)\times\GL(n,q)$ acts on the vector space
  $A=\Mat$ via left-right multiplication:
  $$
  \CT\ni(g,h):x\to gxh^{-1}
  $$
  The $\CT$-orbit of a matrix $x\in A$ is the set of matrices with the
  same rank.  In particular, $A$ has no nontrivial proper
  $\CT$-invariant subgroup.

  Let $V\subseteq A$ be the set of matrices of rank less than
  $\frac23n$.  It is $\CT$-invariant.
  By \fref{fact:number-of-matrices}
  $$
  |V|\le 4q^{\frac89n^2}
  $$
  Consider any four matrices $a\in A$, $b_1,b_2,b_3\in A\setminus V$.  Let
  $r_i=\rank(b_i)\ge\frac23n$.  \fref{lem:k=3-new} implies that there are
  $\Ordo(1)\frac{\RRR(r_1)\,\RRR(r_2)\,\RRR(r_3)}{q^{n^2}}$
  ways to write $a$ as a sum
  $$
  a=x_1+x_2+x_3
  $$
  with $x_i\in A$, $\rank(x_i)=r_i$.  Moreover, for each $x_i$ there are
  exactly $\frac{|\CT|}{\RRR(r_i)}$ elements of $\CT$
  which send $b$ to $x_i$.  Hence there are at most
  $$
  \Ordo(1)\frac{\RRR(r_1)\,\RRR(r_2)\,\RRR(r_3)}{q^{n^2}}\:
  \frac{|\CT|}{\RRR(r_1)}\,\frac{|\CT|}{\RRR(r_3)}\,\frac{|\CT|}{\RRR(r_3)}
  = \frac{|\CT|^3}{\Omega(1)q^{n^2}}
  $$
  triples $(t_1,t_2,t_3)\in\CT^3$
  such that $t_1b_1+t_2b_2+t_3b_3=a$.
  We can apply
  \fref{cor:growth-with-endomorphisms} to the $\CT$-action on $A$ and
  the subset $V$, with parameters
  $$
  l=3,\quad\quad
  \alpha=C\,q^{n^2}, \quad\quad \beta=4\,q^{\frac89n^2}, \quad\quad
  \epsilon =\frac1{90}
  $$
  for some sufficiently small positive $C$.  By decreasing $C$ if
  needed we make sure that
  $$
  \alpha^\epsilon\beta<q^{\frac9{10}n^2}
  $$
  Now, with the notation of
  \fref{cor:growth-with-endomorphisms}
  we set
  $$
  \nu = \nu(\epsilon,l),
  \quad\quad
  \mu = \nu\cdot\left\lceil\frac2C\right\rceil.
  $$
  Let $X_1\dots,X_\mu$ be as in the statement.
  We partition the sequence $X_1\dots,X_\mu$ into
  $\lambda$ subsequences of size $\nu$:
  $$
  X_{j\nu+1},\dots,X_{(j+1)\nu}
  \quad\quad
  \text{for }j=0,\dots,\left\lceil\frac2C\right\rceil-1
  $$
  We apply \fref{cor:growth-with-endomorphisms} to each of these
  subsequences.
  $|X_i|\ge q^{\frac9{10}n^2}>1$ implies $|X_i|\ge2$ for all $i$,
  hence the only $\CT$-invariant subgroup of size at least
  $\min\big(|X_1|,\dots,|X_\mu|\big)$
  is $A$ itself.
  Therefore \fref{cor:growth-with-endomorphisms}
  gives us elements
  $t_1,\dots,t_\mu\in\CT$
  such that
  $$
  Y_j=t_{j\nu+1}X_{j\nu+1}+\dots+t_{(j+1)\nu}X_{(j+1)\nu}
  $$
  has size at least $\alpha$ for all $j$.
  Then \fref{cor:constant-growth-or-filling}
  gives us elements
  $s_1,\dots,s_\lambda\in\CT$ such that
  $$
  A=s_1Y+\dots+s_mY= \sum_{i=1}^\mu\sum_{j=1}^m(t_is_j)X
  $$
\end{proof}

\printbibliography

\end{document}